\documentclass[11pt,oneside]{amsart}

\usepackage[T1]{fontenc}
\usepackage[utf8]{inputenc}
\usepackage{stmaryrd}
\usepackage{xcolor}
\usepackage{MnSymbol}
\usepackage[english]{babel}
\usepackage[english=american,french=guillemets]{csquotes}
\MakeOuterQuote{"}
\usepackage[all]{xy}
\usepackage[shortlabels]{enumitem}
\setlist{nosep}
\usepackage{microtype}
\usepackage[style=alphabetic,backend=biber,sorting=nyt,giveninits=true,backend=biber]{biblatex}
\renewbibmacro{in:}{}
\usepackage[colorlinks=true, breaklinks=true, urlcolor= black, linkcolor= black, citecolor= black, 
bookmarksopen=true,linktocpage=true,plainpages=false,pdfpagelabels]{hyperref}
\usepackage[capitalize]{cleveref}

\newtheorem{theorem}{Theorem}[section]
\newtheorem{corollary}[theorem]{Corollary}
\newtheorem{lemma}[theorem]{Lemma}
\newtheorem{proposition}[theorem]{Proposition}

\newtheorem{fact}[theorem]{Fact}
\newtheorem{claim}{Claim}[theorem]

\newtheorem{thm}[theorem]{Theorem}
\newtheorem{cor}[theorem]{Corollary}

\newtheorem{prop}[theorem]{Proposition}

\theoremstyle{definition}

\newtheorem{definition}[theorem]{Definition}
\newtheorem{notation}[theorem]{Notation}
\newtheorem{convention}[theorem]{Convention}
\newtheorem{remark}[theorem]{Remark}
\newtheorem{example}[theorem]{Example}
\newtheorem{hypothesis}[theorem]{Hypothesis}

\crefformat{hypothesis}{Hypothesis~#2#1#3}
\crefrangeformat{hypothesis}{Hypotheses~#3#1#4 to~#5#2#6}
\crefmultiformat{hypothesis}{Hypotheses~#2#1#3}%
{ and~#2#1#3}{, #2#1#3}{ and~#2#1#3}
\newenvironment{hyp}{\begin{hypothesis}}{\end{hypothesis}}

\newtheorem{rem}[theorem]{Remark}

\newcommand{\Aa}{\mathbb{A}}

\newcommand{\Pp}{\mathbb{P}}

\newcommand{\Zz}{\mathbb{Z}}

\newcommand{\cL}{\mathcal{L}}

\newcommand{\fU}{\mathfrak{U}}
\newcommand{\fd}{\mathfrak{d}}

\newcommand{\rg}{\mathrm{rg}}
\newcommand{\Lrg}{\cL_{\rg}}
\newcommand{\PTC}[1][T]{\mathrm{P}#1\mathrm{C}}
\newcommand{\SPTC}[1][I]{\mathrm{P}T\mathrm{C}_{#1}}
\newcommand{\PIC}[1][I]{\ensuremath{\mathrm{P}#1\mathrm{C}}}
\newcommand{\bSPTC}{\PIC_\fd}
\newcommand{\substr}{\leq}
\newcommand{\sm}[1]{#1_{\mathrm{sm}}}
\newcommand{\subsel}{\preccurlyeq}
\newcommand{\supsel}{\succcurlyeq}
\newcommand{\ACF}{\mathrm{ACF}}
\newcommand{\imp}{\ensuremath{\Rightarrow}}

\newcommand{\sminus}{\smallsetminus}

\newcommand{\compl}[1]{\widehat{#1}}
\newcommand{\sep}[1]{#1^{\mathrm{s}}}
\renewcommand{\mid}{:}
\newcommand{\comp}{\circ}

\newcommand\isom{\simeq}

\newcommand{\acl}{\mathrm{acl}}

\newcommand{\dcl}{\mathrm{dcl}}

\newcommand{\tp}{\mathrm{tp}}
\newcommand{\TP}{\mathcal{S}}

\newcommand\qfequiv{\equiv^{\mathrm{qf}}}
\newcommand{\qftp}{\mathrm{qftp}}

\newcommand{\Gal}{\mathrm{Gal}}
\newcommand{\bdn}{\mathrm{bdn}}

\newcommand{\lindep}{\downfree^{\ell}}
\newcommand{\aindep}{\downfree^{\mathrm{a}}}
\newcommand{\ind}{\downfree}
\newcommand{\tensor}{\otimes}
\newcommand{\alg}[1]{#1^{\mathrm{a}}}
\newcommand{\res}{\mathrm{res}}
\newcommand{\ignore}[1]{}
\newcommand{\restr}[2]{\left.#1\right|_{#2}}

\newcommand\aut{\mathrm{aut}}

\newcommand\Ld{\cL_\fd}
\newcommand\Td{T'}
\newcommand\Th{\mathrm{Th}}
\newcommand\RV{\mathrm{RV}}
\newcommand\val{\mathrm{val}}
\newcommand\Val{\mathcal{O}}
\newcommand\m{\mathfrak{m}}
\newcommand\K{\mathrm{K}}
\newcommand\R{\mathrm{R}}

\newcommand\PAC{\mathrm{PAC}}
\newcommand\PRC{\mathrm{PRC}}
\newcommand\PpC{\mathrm{P}p\mathrm{C}}
\newcommand\RCF{\mathrm{RCF}}
\newcommand\pCF{p\mathrm{CF}}
\newcommand\ACVF{\mathrm{ACVF}}
\newcommand\DCF[1]{\mathrm{DCF}_{#1}}
\newcommand\Spec{\mathrm{Spec}}
\newcommand\trdeg{\mathrm{trdeg}}
\newcommand\Vg{\Gamma}
\renewcommand\k{\mathrm{k}}
\newcommand\NTP{\mathrm{NTP}}

\renewcommand\epsilon{\varepsilon}

\bibliography{short,biblio}

\title{Pseudo \(T\)-closed fields}
\author{Samaria Montenegro}
\address{Samaria Montenegro, Escuela de
    Matem\'atica-CIMPA, Universidad de Costa Rica, San José, Costa Rica}
\thanks{The first author was partially
supported by Proyecto Semilla 821-C0-464, Vicerrector\'ia de Investigaci\'on, UCR}
\author{Silvain Rideau-Kikuchi}
\thanks{The second author was partially supported by GeoMod AAPG2019 (ANR-DFG), Geometric and Combinatorial Configurations in Model Theory}
\address{Silvain Rideau-Kikuchi, DMA, École normale supérieure, Université PSL, CNRS, 75005 Paris, France}

\begin{document}
\begin{abstract}
Pseudo algebraically closed, pseudo real closed, and pseudo $p$-adically closed
fields are examples of unstable fields that share many similarities, but have
mostly been studied separately. In this text, we propose a unified framework
for studying them: the class of pseudo $T$-closed fields, where $T$ is an
enriched theory of fields. These fields verify a "local-global" principle for
the existence of points on varieties with respect to models of $T$. This
approach also enables a good description of some fields equipped with multiple
$V$-topologies, particularly pseudo algebraically closed fields with a finite
number of valuations. One important result is a (model theoretic) classification
result for bounded pseudo $T$-closed fields, in particular we show that under
specific hypotheses on $T$, these fields are NTP$_2$ of finite burden.
\end{abstract}

\maketitle

  

\section{Introduction}

One of the main examples (and motivation) for the development of unstable
theories, especially simple theories, was that of pseudo algebraically closed
fields (PAC fields), and especially those that have a small Galois group ---
\emph{i.e.} those with finitely many extensions of any given degree \(d\),
usually called \emph{bounded} fields. This class of fields was introduced by Ax
in \cite{Ax} to describe the theory of finite fields. A $\PAC$ field is a field
$K$ such that any geometrically integral variety over $K$ admits a $K$-rational
point; in other words, they are existentially closed (in the language of rings)
in any regular extension.

This notion has proven to be very fruitful, whether from the arithmetic point of
view (see for example \cite{FriJar-FA}) and the model theoretic point
of view where it continues to be the source of many examples of moderate
structures, see for example \cite{Cha-OmPAC}. It was eventually reinterpreted as a special case of a wide range of
local-global principles for the existence of rational points with the emergence 
of the pseudo real closed fields ($\PRC$) of Prestel \cite{Pre-PRC}  (the
existence of smooth points in any real closure implies the existence of rational
points) and pseudo $p$-adically closed fields ($\PpC$) \cite{Gro, HarJar-PpC}
(the analogous notion for $p$-adic valuations) and, eventually, the pseudo
classically closed fields of Pop \cite{Pop-PCC} or the regularly \(T\)-closed
fields of Schmid \cite{Sch-RegT}.

Despite the great activity in model theory surrounding $\PAC$ fields, for quite
some time, the model theoretic study of these fields was mostly confined to
questions of completeness and decidability. In \cite{Mon-NTP2}, the first author
initiated the classification (in the sense of Shelah) of bounded \(\PRC\) and
\(\PpC\) fields by showing, among other things, that bounded \(\PRC\) and
\(\PpC\) fields are \(\NTP_2\) of finite burden. This was followed by further
work on imaginaries \cite{Mon-EIPRC} and \cite{MonRid-PpC} and definable groups
\cite{MonOnsSim}.

Notwithstanding an apparent similarity in techniques, $\PRC$ fields and $\PpC$
fields are studied separately in those works. The purpose of this text is to
propose a common framework for the study of these fields. This has, of course,
already been attempted before, and our approach is closely related to that of
Schmid \cite{Sch-RegT}: given a theory \(T\) of (large) fields, we say that a
field \(K\) is pseudo \(T\)-closed if every geometrically integral variety over
\(K\) with a smooth point in every model of \(T\) containing \(K\) has a point
in \(K\) (see \cref{def PTC}). This class generalizes several of the notions
that appear in the literature (see \cref{Ex PTC}).

What sets this work apart, however, is that we aim for a more model theoretic
flavor both in our definitions and in our goals. In particular, we also aim to
describe pseudo algebraically closed fields endowed with a finite number of
valuations, a family of structure that has, to the best of our knowledge escaped
scrutiny so far, despite its apparent similarities with the aforementioned
examples.

One striking phenomenon that seems to arise in the study of all these
"local-global" principles is the fact that an apparently arithmetic hypothesis,
phrased uniquely in terms of existence of rational points in varieties has
strong topological implications on the density of points for \emph{a priori}
unrelated topologies. In other words, existential closure as a topological field
comes for free from existential closure as a field. This was first observed for
\(\PRC\) fields by Prestel \cite{Pre-PRC}, and was later generalized to other
classes. However, in all of these examples, that behavior is not completely
surprising since the topologies considered can be defined in the field structure
alone. Somewhat more surprisingly, this is also true of algebraically closed
fields by the historic theorem of Robinson \cite{Rob-ACVF} which states that
algebraically closed fields are existentially closed as valued fields despite
the valuation very much not being definable. This was generalized much later to
\(\PAC\) fields by Kollar \cite{Kol-PAC} and was also noticed by Johnson
\cite{Joh-MultVal} who showed that this happens as well when adding independent
valuations to a Henselian valued field.

This stronger form of existential closure has been shown to have model theoretic
consequences \cite{GeyJar,Hon-PACV}, and plays a central (albeit somewhat
hidden) role in the first author's work. It is therefore natural, given a set
\(I\) of completions of \(T\) --- which we think of as "localizations of $T$"
and come with their topologies --- to define the notion of \(I\)-pseudo
\(T\)-closed fields (see \cref{PTCI}) which isolates this stronger notion of
existential closure. Two obvious questions then arise:
\begin{itemize}
\item Are all pseudo \(T\)-closed fields \(I\)-pseudo \(T\)-closed?
\item Are \(I\)-pseudo \(T\)-closed fields interesting from a model theoretic
perspective?
\end{itemize}
The first half of this text gives a positive answer to the first question
(under some finiteness hypothesis), showing that earlier instances of this
phenomenon were not isolated incidents and that, indeed, \(I\)-pseudo
\(T\)-closure is essentially just independence of the involved topologies on top
of pseudo \(T\)-closure --- \emph{cf.} \cref{free ec} for a precise statement.
For example, \(\PAC\) fields remain existentially closed when endowed with
independent valuations.

The second half of the text aims at giving a partial answer to the second
question by showing that, at least in the bounded perfect case, \(I\)-pseudo
\(T\)-closed display tame model theoretic behavior, in line with what we could
expect from generalized \(\PAC\) fields. The presence of topologies is an
obvious obstruction to the simplicity of these fields, but we show that they are
essentially the only obstructions and that they are essentially only as
complicated as their "localizations". For example, we show that their burden is
the sum of the burden of the "localizations" (\cref{bdn PTC}).

Paired with the positive answer to the first question, this provides a wealth of
new examples of finite burden, strong and \(\NTP_2\) fields. One of the most
natural finite burden example is the aforementioned case of bounded perfect
\(\PAC\) fields with several independent valuations. In the appendix, we also
explain how to get rid of the independence assumption.\medskip

The organization of the text is as follows. In section \ref{Preliminaries}, we
give some preliminaries on burden, forking, large fields and \(V\)-topologies.
In section \ref{PTC}, we define the class of pseudo \(T\)-closed fields for a
given theory \(T\) of (enriched) large fields and give several equivalent
definitions in terms of "local-global" principle and existential closure. We
then proceed to prove an approximation theorem for rational points of varieties
(Theorem \ref{approx var}) under some finiteness hypothesis.

In section \ref{IPTC}, we define, given a set \(I\) of completions of \(T\), the
class of \(I\)-pseudo \(T\)-closed fields, and we give several equivalent
definitions in terms of approximation properties and existential closure. We
also prove that this is an elementary class under some mild hypotheses that will
always be verified when the theories in \(I\) are theories of Henselian valued
(or real closed) fields. In section \ref{BoundedPTC}, we initiate the model
theoretic study of perfect bounded  \(I\)-pseudo \(T\)-closed fields. We show
that the theory is model complete once the Galois group is fixed, and types are
given by the isomorphism type of their algebraic closures (\cref{char elem}). We
then deduce that these fields enjoy a form of quantifier elimination "up to
density" (\emph{cf.} \cref{Density2}).

This last property can also be naturally rephrased as an "open core" property
that is typical of generic topological structure: any open definable set is
quantifier free definable (\cref{Density}). Note that our methods are, in fact,
more general and also allow proving such "open core" properties for topological
fields with additional non-topological generic structure.

In section \ref{Burden}, generalizing work of the first author, we prove that,
although the presence of topologies is an obvious obstruction to the simplicity
of these fields, it is essentially the only obstruction. For example,
3-amalgamation holds up to quantifier free obstructions (\emph{cf.}
\cref{Amalgamation}). Along with the earlier density result, this allows us to
compute burden in perfect bounded \(I\)-pseudo \(T\)-closed fields (\cref{bdn
PTC}) and characterize forking (\cref{char fork}).

We conclude with a short appendix describing how, although \(I\)-pseudo
\(T\)-closure forces all the considered topologies to be independent, following
an approach of Johnson \cite{Joh-MulACVF}, we can also describe what happens for
dependent valuations.

We thank the anonymous referee for his careful reading and insightful remarks on an earlier version of this text.

\section{Preliminaries} \label{Preliminaries}

In this section we will recall some notions (and fix some notation) that will be used throughout the paper.

\subsection{Model Theory}

Most of the structures in this paper, if not all, will be fields that we
consider in (some expansion of) the ring language $\Lrg := \{+,-,\cdot,0,1\}$.

\begin{notation} Let $\cL$ be a language and $M$ an \(\cL\)-structure. We write $\acl$ and $\dcl$ for the model theoretic algebraic and definable closure in $M$. If $a$ and $b$ are tuples in $M$ and $A \subseteq M$, then we say that $a \equiv_A b$ ($a \qfequiv_A b$) if $a$ and $b$ have the same (quantifier-free) type over $A$. We will write $\TP(A)$ for the space of $\cL$-types with parameters in $A$.
\end{notation}

\begin{definition}
If \(K\) is a field (considered as a structure in some expansion of \(\Lrg\)),
and $X$ is a $\cL(K)$-definable set, we will write $\dim(X)$ for the algebraic
dimension of $X$:
\[\dim(X):= \max\{\trdeg(a/K): a \in X(M)\text{ where }K\subsel M\}.\]
\end{definition}

\begin{definition}
Let $X$ be definable in some \(\cL\)-structure \(M\) and \(\kappa\) a cardinal.
\begin{enumerate}
    \item An \emph{inp-pattern} of depth \(\kappa\) in $X$  consists of
    definable subsets \(\phi_l(M,a_{lj})\) of \(X\), with \(l < \kappa\) and
    \(j<\omega\), such that:
    \begin{itemize}
    \item for every \(l<\kappa\), \((\phi_l(x,a_{lj}))_{j<\omega}\) is
    \(k_l\)-inconsistent, for some \(k_l\geq 1\);
    \item for every \(f : \kappa \to \omega\), \((\phi_l(x,a_{l f(l)}))_{l<\kappa}\) is consistent.
    \end{itemize}
    \item The \emph{burden} of $X$, denoted $\bdn(X)$ is greater or equal to \(\kappa\), if there exists an inp-pattern of depth \(\kappa\) in $X$ and
    \[\bdn(X) = \sup\{\kappa \mid \bdn(X)\geq \kappa\}.\]
    To keep track of finer estimates, burden can be computed in \(\mathrm{Card}^\star\), \emph{cf.} \cite[Definition~1.27]{Tou-bdn}.
    \item The burden of \(M\), or equivalently the burden of \(\Th(M)\), is the burden of the home sort (provided \(M\) has a unique dominant sort).
\end{enumerate}
\end{definition}

\begin{remark} \label{burden}
\begin{enumerate}
    \item 
By Lemma 2.2 of \cite{Che-NTP2} if $\phi_{l}(M,a_{lj})$ is an inp-pattern, then
you can always assume that the sequences $(a_{lj})_{j<\omega}$ are mutually
indiscernible over any given $A \subseteq M$, meaning that $(a_{lj})_{j<\omega}$
is indiscernible over $A\cup \{a_{l'j} \mid l' \neq l, j<\omega\}$.

\item By Lemma 7.1 of \cite{Che-NTP2} if $\phi_{l,0}(M,a_{l,0, j}) \lor \phi_{l,1}(M,a_{l,1, j})$ is an inp-pattern, then $\phi_{l,f(l)}(M,a_{l,f(l), j})$ is an inp-pattern for some $f: \lambda \rightarrow\{0,1\}$. 

\item If $X,Y$ are definable, then $\bdn(X)+ \bdn(Y) \leq \bdn(X \times Y)$.

 \item If \(X\),\(Y\) are definable, then \(\bdn(X\cup Y) = \max(\bdn(X),\bdn(Y))\).
  \item If \(f: Y \to X\) is definable and surjective, then \(\bdn(Y) \geq \bdn(X)\).
    \item If \(f: Y \to X\) is definable and finite-to-one, then \(\bdn(X) \geq \bdn(Y)\).
\end{enumerate}

\end{remark}

\begin{definition}
A theory $T$ is called \emph{strong} if there are no inp-patterns of infinite
depth in its models. It is \(\NTP_2\) if there is a (cardinal) bound on the
depth of inp-pattern in its models.
\end{definition}

The burden of a real closed field (in the ring language, or equivalently the
ordered ring language) is \(1\). In examples, we will also consider valued
fields. Given a valued field \((K,v)\), we denote \(\Val\) its valuation ring
with maximal ideal \(\m\), \(\k = \Val/\m\) its residue field, \(\Gamma =
v(K^\times)\) its valuation group and \(\RV = K/(1+\m)\) its leading terms. We
refer the reader to \cite{Tou-bdn} for a very throughout introduction to their
model theory and the proof of the following burden computations:

\begin{fact}[{\cite[Corollary~3.13 and Theorem~3.21]{Tou-bdn}}]\label{bdn VF}
Let \((K,v)\) be a valued field.
\begin{enumerate}
  \item If \(K\) eliminates quantifiers relatively to \(\RV_1\), is elementary
  equivalent to a maximally complete valued field and the \(\k^{\star}/\k^{\star n}\)
  are finite, then
  \[\bdn(K) = \max\{\bdn(\k),\bdn(\Vg)\}.\]
  \item If \(K\) eliminates quantifiers relatively to \(\bigcup_n\RV_n\), is
  elementary equivalent to a maximally complete valued field and \(K\) is
  finitely ramified --- that is, for all \(n\geq 1\), \((0,v(n))\) is finite ---
  then \[\bdn(K) = \max\{\aleph_0\cdot \bdn(\k),\bdn(\Vg)\}.\]
\end{enumerate}
These equalities hold resplendently for \(\Vg\)-expansions of \(\k\)-expansions.
\end{fact}

In \cite{Tou-bdn}, the results are stated under stronger assumptions, but as
noted there, they hold at this level of generality ---  up to the fact that
Touchard considers unramified mixed characteristic fields (and not finitely
ramified ones), but the proof there also applies since the higher residue rings
\(\R_n = \Val/n\m\) are also interpretable in \(\k\) in that case.

\begin{definition}\label{forking}
Let \(M\) be a (sufficiently saturated) \(\cL\)-structure and \(A\leq M\).
\begin{enumerate}
 \item We say that the formula $\phi(x,a)$ \emph{divides over $A$} if there
 exists an indiscernible sequence $(a_j)_{j \in \omega}$ over $A$ such that
 $a_0= a$ and $\{\phi(x,a_j): j \in \omega\}$ is inconsistent.
\item We say that the formula $\phi(x,a)$ \emph{forks over $A$} if there is an
integer $m$ and formulas $\psi_j(x,a_j)$ for $j <m$ such that $\phi(x,a) \vdash
\bigvee_{j <m}\psi_j(x,a_j)$ and $\psi_j(x,a_j)$ divides over $A$ for every
$j<m$.   
\item A partial type $p$ \emph{forks} (respectively \emph{divides}) \emph{over
$A$} if there is a formula in $p$ which forks (respectively divides) over $A$.
\item We say that $A$ is \emph{an extension base} if for all tuples $a$ in $M$,
$tp(a /A)$ does not fork over $A$.

\end{enumerate}
By Theorem 1.2 of \cite{CheKap}, in an NTP$_2$ theory, forking equals dividing
over extension bases.
\end{definition}

\subsection{Large fields}

 In this text, by a variety \(V\) over a field \(K\), we mean a separated
 integral scheme of finite type over \(K\). We denote by \(V(K)\) the set of its
 \(K\)-rational points. We also denote by $V_{sm}(K)$ the set of smooth
 $K$-rational points.
 
\begin{remark}[\emph{cf.}
{\cite[\href{https://stacks.math.columbia.edu/tag/01V7}{Tag
01V7}]{stacks-project}}]\label{smoothpoint} A point \(x\in V(K)\) is smooth if
there exists an affine neighborhood \(U\subseteq V\) of \(x\) and an étale
morphism \(U \to \Aa^d\) --- that is \(U\) is isomorphic to \(\Spec(S)\) where
\(S = K[x_1,\ldots,x_{c+d}]/(f_1,\ldots,f_c)\) and the matrix \((\frac{d
f_i}{d x_j})_{i\leq c,j\leq c}\) is invertible in \(S\).
\end{remark}

\begin{definition}
A field $K$ is \emph{large} if for every irreducible variety $V$, if
$\sm{V}(K)\not = \emptyset$, then $V(K)$ is Zariski dense in \(V\) ---
equivalently, by compatness, \(K(V)\) can be \(\Lrg(K)\)-embedded in some
\(K^\star \succcurlyeq K\).
\end{definition}

We know that algebraically closed fields,  real closed fields, $p$-adically
closed fields and, more generally, Henselian fields are large; as well as pseudo
algebraically closed fields, pseudo real closed fields and pseudo \(p\)-adically
closed fields, and further generalizations of these notions.\smallskip

Finally, given two extensions \(M\) and \(K\) of a field \(F\), we write $M
\lindep_{F} K$ if the field $M$ is linearly disjoint from $K$ over $F$ and \(L
\aindep_{F} K\) if  $M$ and $K$ are algebraically independent over $F$. We also
denote by \(\alg{K}\) the algebraic closure of \(K\) and by \(\sep{K}\) its
separable closure. Recall that a variety \(V\) over \(K\) is geometrically
integral (equivalently absolutely irreducible defined over \(K\)) if the
extension \(K \leq K(V)\) is regular --- that is \(K(V)\lindep_K \alg{K}\).

\subsection{V-topological fields}

\begin{definition}
Let \((K,\tau)\) be a (non-discrete) topological field.
\begin{enumerate}
\item A subset \(S\subseteq K\) is said to be \emph{bounded} if for every  for every open neighbourhood of the identity 
\(U\subset K\), there exists \(x\in K^\times\) such that \(xS = \{x\cdot y \mid
y \in S\} \subseteq U\) --- equivalently, \(\{x\in K \mid xS \subseteq U\}\) is
a (non-trivial) neighborhood of zero.
\item A subset \(S\subseteq K\) is said to be \emph{bounded away from \(0\)} if there
exists a neighborhood \(U\) of zero such that \(S\cap U = \emptyset\).
\item The topology \(\tau\) is said to be a \emph{\(V\)-topology} if, whenever \(S^{-1}
= \{x^{-1}\mid x\in S\}\) is bounded away from \(0\), \(S\) is bounded.
\item The topology \(\tau\) is said to be \emph{Henselian} if every integer \(n \geq
1\) there is a neighborhood \(U\) of zero such that every polynomial in
\(X^{n+1} +X^n +U[X]^{n-1}\) has a zero in \(K\) --- where \(U[X]^{n-1}\) denotes
the set of polynomials of degree at most \(n-1\) with coefficients in \(U\).
\end{enumerate}
\end{definition}

By a theorem of Dürbaum-Kowalsky, and independently Fleischer, a topology is a
\(V\)-topology if and only if it is induced by either an order or a valuation. We
refer the reader to \cite{PreZie} for further detail on \(V\)-topologies.

In this text, the most useful consequence of Henselianity (and an equivalent
characterization) is the inverse function theorem:

\begin{fact}[{\cite[Proposition~2.8]{HalHasJah}}]\label{InvFuncThe}
Let \((K,\tau)\) be a Henselian \(V\)-topolo\-gical field and \(f : V\to W\) be an
étale morphism of varieties. Then any \(x \in V(K)\) admits a neighborhood \(U\)
such that \(\restr{f}{U}\) is a homeomorphism onto its open image.
\end{fact}

\begin{definition}
A \(V\)-topology on a field \(K\) is said to be definable if it admits a uniformly definable basis --- equivalently if there exists one definable bounded open set.
\end{definition}

\begin{remark}\label{bounded cover}
In fact, by (the proof of)
\cite[Lemma~B.2]{EngPre}, if \(\tau\) is a definable \(V\)-topology on \(K\),
then there exists an open bounded definable neighborhood \(U\) of \(0\) such that \(K
= U\cup(U\sminus \{0\})^{-1}\).
\end{remark}

For more details on definable \(V\)-topologies,
we refer the reader to \cite[Section~3]{HalHasJah}.

\section{The class of pseudo \texorpdfstring{\(T\)}{T}-closed fields}\label{PTC}

Let \(\cL\) enrich the language of rings, \(T\) be an \(\cL\)-theory of large fields and \(K\) be an \(\cL\)-structure that is a field --- we will refer to those as \(\cL\)-fields from now on. We write \(T_K\) for the theory of models of \(T\) containing \(K\).
\begin{definition}
\begin{itemize}
\item We say that a field extension \(K\leq F\) is \emph{totally \(T\)} if, for
every \(M\models T_{K}\), \(F\) can be \(\Lrg(K)\)-embedded in some
\(M^\star\supsel M\).

\item We say that a variety \(V\) over \(K\) is \emph{totally \(T\)} if, for every \(M\models T_{K}\), \(\sm{V}(M)\neq \emptyset\).

\end{itemize}
\end{definition}

Observe that if \(V\) is a geometrically irreducible variety over \(K\), then \(\sm{V}\) is Zariski open in \(V\) (\cite[\href{https://stacks.math.columbia.edu/tag/056V}{Tag 056V}]{stacks-project}). In particular \(\sm{V}(K(V))\neq\emptyset\). It follows that, since models of \(T\) are large, \(V\) is totally \(T\) if and only if \(K\leq K(V)\) is totally \(T\).

\begin{rem}
Observe that if $T$ is $\RCF$ in $\Lrg$, we have that $K \leq F$ is totally
$T$  if and only if it is a totally real extension, meaning that every order on
$K$ extends to some order on $F$. Similarly, if $T$ is $\pCF$ in
$\Lrg$, then $K \leq F$ is totally $T$ if every $p$-adic valuation on $K$
extends to some $p$-adic valuation on $F$.
\end{rem}

\begin{proposition}\label{equiv PTC}
The following are equivalent:
\begin{enumerate}[(i)]
\item for every geometrically integral (affine) totally \(T\) variety \(V\) over \(K\), we have \(V(K) \neq \emptyset\);
\item for every geometrically integral (affine) totally \(T\) variety \(V\) over \(K\), \(V(K)\) is Zariski dense in \(V\);
\item Any regular totally \(T\) extension \(K\leq F\) is \(\Lrg\)-existentially closed.
\end{enumerate}
\end{proposition}

\begin{proof} We obviously have that (ii) implies (i).
\begin{itemize}
\item[(i)\imp(iii)] Let \(F\) be as in (iii) and let \(\phi(x)\) be some
quantifier free \(\Lrg\)-formula such that \(F \models \exists x \phi(x)\).
Adding existential quantifiers, we may assume that \(\phi\) is of the form
\(\bigwedge_{i < n} P_i(x) = 0\) where \(P_i\in K[x]\). Let \(a\models \phi\) in
\(F\) and \(V\) be the locus of \(a\) over \(K\). Since \(K\leq K(a) \leq F\) is
regular, \(V\) is geometrically integral and, by (\cite[\href{https://stacks.math.columbia.edu/tag/056V}{Tag 056V}]{stacks-project}), \(\sm{V}(F) \neq \emptyset\). Moreover, by hypothesis on \(F\), for
every \(M\models T_K\), there exists \(M^\star\supsel M\) such that \(F\leq
M^\star\). So \(\sm{V}(M^\star) \neq \emptyset\) and hence \(\sm{V}(M) \neq \emptyset\). By (i), we find \(c\in V(K)\). In particular, \(K\models
\bigwedge_i P_i(c) = 0\).

\item[(iii)\imp(ii)] Let \(V\) be as in (ii). Any \(M\models T_K\) is large and
hence, if \(\sm{V}(M)\neq \emptyset\), then some \(M^\star \supsel M\) contains
a \(K\)-generic point \(a\) --- that is \(a\) is not in any proper sub-variety
of \(V\) over \(K\). In other terms, \(F := K(V)\) is \(\Lrg(K)\)-embeddable in
\(M^\star\). By (iii), \(F\) is \(\Lrg(K)\)-embeddable in \(K^\star\supsel K\),
equivalently, \(V(K)\) is Zariski dense in \(V\).\qedhere
\end{itemize}
\end{proof}

\begin{definition}\label{def PTC}
The \(\cL\)-field \(K\) is said to be \emph{pseudo \(T\)-closed} (\(\PTC\)) if the statements of \cref{equiv PTC} hold.
\end{definition}

\begin{remark}
\begin{enumerate}
    \item It follows from \cref{equiv PTC}.(ii), that \(\PTC\) fields are large.
    \item The class of \(\PTC\) fields is inductive in \(\cL\), \emph{cf.} \cref{SPTC elem}.
\end{enumerate}
\end{remark}

\begin{example} \label{Ex PTC}
For different choices of \(T\), the following are examples of \(\PTC\) fields :
\begin{itemize}
    \item If \(T = \ACF\) in \(\Lrg\), then \(\PTC\) fields are exactly \(\PAC\) fields.
    \item If \(T = \RCF\)  in \(\Lrg\), then \(\PTC\) fields are exactly \(\PRC\) fields.
    \item If \(T = \RCF_<\) in \(\Lrg\cup\{<\}\), then \(\PTC\) fields are exactly
    \(1-\PRC\) fields, \emph{i.e.} \(\PRC\) fields whose only order is \(<\).
    The equivalence of those two classes is not obvious from the definition and
    relies crucially on \cref{PTC dense} to show the only order on a \(\PTC\)
    field is \(<\).
    \item Let \(\cL = \Lrg\cup\{<_1,\ldots,<_n\}\) and \(T  = \bigvee_i \RCF_{<_i}\), whose models are the \(\cL\)-structures \(M\) with \(\restr{M}{\cL_i}\models \RCF_{<_i}\), for some \(i\). Then \(\PTC\) fields where the \(<_i\) define distinct orders are exactly \(n-\PRC\) fields. This equivalence relies again on \cref{PTC dense}.
    \item If \(T = \pCF\) in \(\Lrg\), then \(\PTC\)-fields are exactly \(\PpC\) fields.
    \item For \(i < n\), let \(\cL_i\) be copies of Macintyre's language sharing the ring language and \(\pCF_{v_i}\) the \(\cL_i\)-theory of \(p\)-adically closed fields. Then \(\PTC\) fields whose \(\cL_i\)-structure induces distinct \(p\)-adic valuations are exactly \(n-\PpC\) fields. Once again the equivalence relies on \cref{PTC dense}.
    \item Fehm's pseudo \(\mathcal{S}\)-closed fields \cite{Feh-PCC} also fit in this framework.    
\end{itemize}
\end{example}

\begin{lemma}\label{better PTC}
Let \(F, K\) be \(\cL\)-fields, \(F\leq M\) be regular and totally \(T\), and \(f : F \to K\) be an \(\cL\)-embedding. Then \(K\leq K\tensor_F M\) is regular and totally \(T\).
\end{lemma}

\begin{proof}
We may assume that \(M\) is finite type over \(F\). Then \(M = F(V)\), where
\(V\) is a geometrically integral totally \(T\) variety over \(F\). Let
\(E\models T_K\). Then \(E\models T_F\) and hence \(\sm{V}(E)\neq\emptyset\).
Since \(E\) is large, we find an \(\Lrg(E)\)-embedding \(g : E\tensor_F M \isom
E(V) \to E^\star\supsel E\) --- in particular, we have an \(\Lrg(K)\)-embedding
\(K(V) \isom K\tensor_F M \to E^\star\).
\end{proof}

\begin{convention}
Replacing \(T\) by the theory whose models are either models of \(T\) or algebraically closed does not change the notion of pseudo \(T\)-closed fields: algebraically closed fields always have rational smooth points in any variety. So, from now on, we will assume that any algebraically closed field is a model of \(T\).
\end{convention}

We will now study the topological properties of \(\PTC\) fields. We fix  \(K\)  a \(\PTC\) field.

\begin{notation}\ 
\begin{itemize}
\item For every \(M\models T_K\), let \(K_M:= M\cap\sep{K}\).
\item For every \(V\)-topology \(\tau\) on \(K\), let \(C_\tau := \{K_M\) that can be endowed with a Henselian \(V\)-topology that induces \(\tau\) on \(K\}\).

\item We say that \(F \in C_\tau\) is minimal if any \(\Lrg(K)\)-embedding \(E \to F\), with \(E\in C_\tau\), is surjective.

\item For every \(V\)-topology \(\tau\) on \(K\), let \(K_\tau = \compl{(K,\tau)}\cap\sep{K}\), where  \( \compl{(K,\tau)}\) is the completion of $(K, \tau)$.
\end{itemize}
\end{notation}

Note that, by convention, \(\alg{K}\models T\) and hence \(\sep{K}\in C_\tau\),
for every \(\tau\). Also, by uniqueness of the Henselian \(V\)-topology on a
non-separably closed field --- cf. \cite[Theorem 7.9]{PreZie} --- if \(\tau\)
and \(\tau'\) are distinct \(V\)-topologies on \(K\), \(C_\tau\cap C_{\tau'} =
\{\sep{K}\}\).

From now on we will assume:

\begin{hyp}\label{H1}
Until the end of \cref{PTC}, we assume :
\begin{enumerate}[label=\textbf{(H)}]
\item For every \(M\models T_K\), \(K_M\) can be endowed with at least one
Henselian \(V\)-topology.
\end{enumerate}
\end{hyp}

\begin{lemma}
For every \(\tau\), \(C_\tau\) admits minimal elements.
\end{lemma}

\begin{proof}
By Zorn's lemma, it suffices to show that any decreasing chain (of non separably
closed fields) \((K_i)_{i\in I}\) in \(C_\tau\) admits a lower bound. Let
\(M_i\models T_K\) such that \(K_i = M_i\cap \sep{K}\). The theory \(T_K
\cup\{\forall x\ P(x) \neq 0 : P\in K[x]\) such that, for some \(i\in I\), \(M_i
\models\forall x\ P(x) \neq 0\}\) is consistent since any finite subset is
realized in an \(M_i\) and let \(M\) be a model. For all \(i\in I\) and \(P\in
K[x]\), if \(K_M \models \exists x\ P(x) = 0\), then we also have \(M_i \models
\exists x\ P(x) = 0\), so \(K_M\) can be embedded in \(M_i\) over \(K\),
\emph{i.e.} it is a lower bound. Moreover, by Hypothesis \textbf{(H)}, \(K_M\)
admits a Henselian \(V\)-topology which extends to the unique Henselian
\(V\)-topology on every \(K_i\). In particular, it induces \(\tau\) on \(K\).
\end{proof}

As in \cite[Proposition\,4.5]{Sch-RegT}, we will see that minimal elements are
unique (up to isomorphism).

\begin{proposition}
\label{PTC dense}
Let \(\tau\) be a \(\mathrm{V}\)-topology on \(K\) and  \(F \in C_{\tau}\) be minimal. Then \(F\) is \(\Lrg(K)\)-homeomorphic to \(K_\tau\), in particular, \(K\) is \(\tau\)-dense in \(F\).
\end{proposition}

\begin{proof}
We start with two intermediary results:
\begin{claim}
There exists a \(\tau\)-continuous \(\Lrg(K)\)-embedding \(\sigma : K_\tau \to F\).
\end{claim}

\begin{proof}
There is a continuous \(\Lrg(K)\)-embedding \(\sigma : \compl{(K,\tau)} \to \compl{(F,\tau)}\), by universality of the completion.
Since \(K\subseteq K_\tau\) is separable and \(F\), being topologically Henselian, is separably closed in \(\compl{(F,\tau)}\), \emph{cf.} \cite[Corollary\,7.6]{PreZie}, we have \(\sigma(K_\tau)\substr \compl{(F,\tau)}\cap \sep{K} = F\).
\end{proof}

\begin{claim}
There exists \(M\models T_K\) and an \(\Lrg(K)\)-embedding \(\rho: K_M \to \compl{(K,\tau)}\).
\end{claim}

\begin{proof}
If no \(K_M\) embeds in \(\compl{(K,\tau)}\), then, by compactness, there exists
a (non-trivial) separable \(P \in K[x]\) with a zero in each \(K_M\) but no zero
in \(\compl{(K,\tau)}\). Then there exists a neighborhood \(U\subseteq K\) of
\(0\) such that \(P(\compl{(K,\tau)})\cap U = \emptyset\). In particular,
\(P(K)\cap U = \emptyset\). It follows that \(P(K)^{-1}\) is bounded and hence
so is \(1-P(0)P(K)^{-1}\). Let \(c\in K^\star\) be such that
\(c(1-P(0)P(K)^{-1})\subseteq U\). Let \(Q(x,y) = P(x)(1-c^{-1}P(y)) - P(0)\in
K[x,y]\). By \cite[Proposition\ 1.1]{HeiPre}, the zero locus \(V\) of \(Q\) is
geometrically integral. For every \(M\models T\), one can check that
\((0,a)\in\sm{V}(M)\), where \(a\in M\) is a zero of \(P\). So there exists
\((b,a) \in V(K)\), \emph{i.e.} \(P(b) = c(1-P(0)P(a)^{-1}) \in U\), a
contradiction.
\end{proof}

Note that, since \(K\leq K_M\) is separable, \(\rho(K_M)\leq K_\tau\). If \(K_M \in C_{\tau'}\), for some \(\tau'\neq \tau\), then \(\tau'\) extends to a Henselian \(V\)-topology on \(K_{\tau}\), distinct from \(\tau\). Moreover, \(K_\tau\) is, by definition, separably closed in its \(\tau\)-completion. It follows by \cite[Improved Theorem\ 7.9]{PreZie}, that \(K_\tau = \sep{K}\), and hence \(F\) \(\Lrg(K)\)-embeds in \(K_\tau\). So we may assume that \(K_M\in C_\tau\).

Now, the \(\Lrg(K)\)-embedding  \(\sigma\comp\rho : K_M\to K_\tau\to F\) must be
surjective by minimality of \(F\) and hence \(\sigma\) is a \(\tau\)-continuous \(\Lrg(K)\)-isomorphism. In particular, \(K\) is \(\tau\)-dense in \(F\). It follows that the identity on \(K\) extends to a (unique) \(\tau\)-continuous \(\Lrg(K)\)-morphism  \(F \to K_\tau\) which is the inverse of \(\sigma\); and \(\sigma\) is indeed an \(\Lrg(K)\)-homeomorphism.
\end{proof}

\begin{notation}
Fix \((\tau_i)_{0\leq i < n}\) distinct $V$-topologies on \(K\) and \(K_i = K_{\tau_i} \in C_{\tau_i}\) minimal. 
\end{notation}

Since, by Stone's approximation theorem (see \cite[Theorem 4.1]{PreZie}), distinct $V$-topologies are independent, we deduce the following approximation result on the affine line:

\begin{corollary}\label{approx aff}
For every non-empty \(\tau_i\)-open \(O_i \leq K_i\), \(\bigcap_i O_i\cap K \neq \emptyset\).
\end{corollary}

Our goal now is to extend this approximation result to arbitrary geometrically integral totally \(T\) varieties. We don't know, however, if we can do so without the following finiteness hypothesis:

\begin{hyp}\label{F1}
Until the end of \cref{PTC}, we assume :
\begin{itemize}
\item[\textbf{(F$_1$)}] There are only finitely many non-separably closed \(K_\tau\), among all possible $\tau$.
\end{itemize}
\end{hyp}

\begin{prop}\label{approx curve} Let \(C\) be a totally \(T\) smooth projective geometrically
integral curve over \(K\). For every non-empty \(\tau_i\)-open
sets \(O_i\subseteq C(K_i)\), \(\bigcap_{i < n} O_i \cap C(K)\) is infinite.
\end{prop}

This statement (and its proof) is inspired by similar statements in two
important and orthogonal subcases: Henselian fields enriched with further
valuations \cite[Theorem~4.1]{Joh-MultVal} (see also
\cite[Proposition~4.2]{HalHasJah}) and PAC fields with a valuation
\cite[Theorem~2]{Kol-PAC}.

\begin{proof} We proceed by induction on \(n\). If \(n = 0\), since \(C(K)\) is
Zariski dense in \(C\), it is, in particular, infinite. Let us now assume \(n >
0\). By induction, \(\bigcap_{i > 0} O_i\cap C(K)\) contains a finite set \(P\)
of size \(d\) larger than the genus of \(C\). By Riemann-Roch, there exists a
non-constant rational function \(f\) on \(C\) over \(K\), whose set of poles is
contained in \(P\) and whose poles are all simple. This gives rise to a morphism
\(f : C \to \Pp^1\) over \(K\) with \(f^{-1}(1:0) \subseteq P \subseteq
\bigcap_{0<i<m} O_i\cap C(K)\) and which is étale above \((1:0)\).

Let \(\tau\) be any \(V\)-topology on \(K\) distinct from \(\tau_0\). If \(\tau
= \tau_i\), let \(O_\tau = O_i\) and otherwise, let \(O_\tau = V(K_\tau)\). By
the inverse function theorem, \cref{InvFuncThe}, for every \(p\in P \subseteq
O_\tau\) there exists a \(\tau\)-open neighborhood \(W_{p,\tau}\subseteq
O_\tau\) of \(p\) such that \(\restr{f}{W_{p,\tau}}\) is a homeomorphism unto
its open image. We may assume that the \(W_{p,\tau}\) do not intersect when
\(p\) varies. Let \(U_\tau := \bigcap_{p\in P} f(W_{p,\tau})\). It is open (and
hence Zariski dense) in  \(\Pp^1(K_\tau)\). For every \(y \in U_\tau\),
\(f^{-1}(y)\) consists of \(d\) distinct elements of \(O_\tau\).

Let \(e\geq 1\) be maximal such that \(\{t\in \Pp(K_0) \mid \exists^\neq y_0
\ldots y_{e-1}\in C(K_0)\ f(y_j) = t\) and \(y_0\in O_0\}\) is infinite. Let
\(D\) be the normalization of the irreducible component of the \(e\)-fold
product of \(C\) over \(\Pp^1\) containing \(y_{< e}\). Then \(K\leq K(D) =
K(y_{< e})\) is regular and \(D\) is geometrically integral. Let \(g : D \to C\)
be the projection on the first coordinate and \(h = f\comp g\). Let \(W_{\tau_0}
:= \{x\in O_0 \mid \exists y\in D(K_0) \mid g(y) = x\) and \(g\) is smooth a
\(y\}\) and \(U_{\tau_0} = f(W_{\tau_0})\). Recall that, \(\tau_0\) being
Henselian, \(f\) and \(g\) are local homeomorphisms at smooth points, so both
sets have non-empty \(\tau_0\)-interior. Shrinking them, we may assume that they
are \(\tau_0\)-open and that \(U_{\tau_0}\) is a ball.

Let \(J := \{\tau \mid K_\tau\neq \sep{K_\tau}\} \cup\{\tau_i \mid i < n\}\),
which is finite by \textbf{(F$_1$)}. By \cref{approx aff}, there exists \(a\in
\bigcap_{\tau\in J} U_{\tau} \cap K\). Composing \(f\) with a degree one
morphism, we may assume that \(a = (1:0)\). For every \(s\in\Aa^1\), let \(B_{s}
\subseteq D\times D\) be defined by \((s^2 u_1 u_2 - v_1 v_2)\comp (h\times h) =
0\), where \(([u_1:v_1],[u_2,v_2])\) are coordinates on \(\Pp_1\times\Pp_1\). By
\cite[Lemma\ 15]{Kol-PAC}, for all but finitely many \(s\), \(B_s\) is
geometrically integral. Applying \cref{approx aff} again, we find \(s\in
K^\times\) such that \(B_s\) is geometrically integral and \((1:s) \in
\bigcap_{\tau\in J} U_{\tau}\). Then, for every \(\tau\in J\) and \(y\in
D(K_\tau)\) with \(h(y) = (1:s)\), \((y,y)\) is a smooth point of \(B_s(K_\tau)\)
--- with \((h\times h)(y,y) \in U_\tau\times U_\tau\). Note that for every
\(M\models T_K\), if \(K_M\) is not separably closed, it contains some
non-separably closed \(K_\tau\) and hence there is a smooth \(M\)-point on
\(B_s\). If \(K_M\) is separably closed, \(B_s\) also has a smooth \(M\)-point,
so \(B_s\) is totally \(T\).

By induction, we can find infinitely many \((y_1,y_2)\in B_s(K)\) with \(h(y_k)
\in \bigcap_{0 < i} U_{\tau_i}\sminus \{(0:1)\}\). Considering the affine
coordinates \(v/u\) on \(\Pp^1\), we see that one of the \(h(y_k)\) is
\(\tau_0\)-closer to \(0\) than \(s\) and hence \(h(y_k) \in U_{\tau_0}\). In
other words, we have found infinitely many \(t \in \bigcap_{i < n}
U_{\tau_i}(K)\) and \((x_{t,i})_{i<e} \in C(K)\) with \(f(x_{t,i}) = t\). By
construction there exist \(x_t\in O_0\) with \(f(x_t) = t\). By maximality of
\(e\), \(x_t\) is distinct from the \(x_{t,i}\) for at most finitely many \(t\).
So we may assume that \(x_{t,0} = x_t \in O_0(K)\). Moreover, for every \(i >
0\), since \(f(x_t) = t \in U_{\tau_i}\), we have \(x_t \in f^{-1}(U_{\tau_i})
\subseteq O_i\) and hence \(x_t \in \bigcap_i O_i(K)\).
\end{proof}

\begin{theorem}\label{approx var}
Let \(V\) be a geometrically integral \(K\)-variety. Let \(O_i\subseteq \sm{V}(K_i)\) be non-empty \(\tau_i\)-open sets, for every \(i < n\). Then \(\bigcap_{i < n} O_i\cap V(K) \neq \emptyset\).
\end{theorem}

\begin{proof}
For every \(i < n\), fix an \(a_i \in O_i \subseteq V(\sep{K})\). By \cite[Lemma~10.1]{JarRaz} there exists a smooth geometrically integral affine \(K\)-curve \(C\subseteq V\) containing the \(a_i\). Let \(U_i = O_i\cap C\) which is a \(\tau_i\)-open non-empty subset of \(C(K_i)\). Let \(\overline{C}\) be a smooth projective model of \(C\). By \cref{approx curve}, there exists \(y\in\bigcap_i U_i\cap\overline{C}(K) \subseteq \bigcap_i O_i\cap V(K)\).
\end{proof}

\begin{remark}
Schmid proved a similar result, \cite[Theorem\ 4.9]{Sch-RegT}, without any
finiteness hypothesis, but requiring that no minimal \(K_M\) is separably closed
--- a case that we certainly do not want to omit if we want to say anything
about valued pseudo algebraically closed fields.

On the other hand, the proof we give here relies heavily on hypothesis
\textbf{(F$_1$)}. Without it, we would have to find points in infinitely many
(uniformly defined) open sets for independent topologies. This is reminiscent of
known approximation theorems (\emph{eg}. \cite{AnsDitFeh-Approx}), but it is not
obvious that they apply here.
\end{remark}

\section{The class of \texorpdfstring{\(I\)}{I}-pseudo \texorpdfstring{\(T\)}{T}-closed fields} \label{IPTC}

Let us fix the following notation for the rest of the text. As before, let
\(\cL\) enrich the language of rings, \(T\) be an \(\cL\)-theory of large
fields.

Let \(I\) be a (potentially infinite) set of theories which eliminate
quantifiers in (disjoint) relational expansions of the language of rings. We
will often denote by \(T_i\) the element \(i\in I\) and \(\cL_i\) its language.
From now on, all \(\cL_i\)-formulas will be assumed to be quantifier free. Let
\(\cL_I = \bigcup_i \cL_i\) and \(T_I = \bigcup_i T_{i, \forall}\). 

For every \(K\models T_{i,\forall}\), we denote by $T_{i,K}$ the theory of models of $T_i$ containing $K$.
We will also assume that any \(M_i\models T_{i,K}\) can be expanded to a model of \(T_K\). When \(I\) is finite, this can always be assumed by replacing \(T\) by \(T\vee\bigvee_i T_i\) whose models are the \(\cL\cup\cL_I\)-structures \(M\) such that \(\restr{M}{\cL}\models T\) or \(\restr{M}{\cL_i}\models T_i\) for some \(i\).

For each \(i\in I\), we will denote by $\acl_i$ and $\dcl_i$ the algebraic and definable closure in models of $T_i$. We will also write $\TP^i(K)$ for the space of (quantifier free) $\cL_i$-types with parameters in $K$. 
We will use the notation $\tp_i(a/A)$ for the type of $a$ over $A$ in the language $\cL_i$.

Let us fix some \(\cL\)-field \(K\models T_I\) and \(M_i \models T_{i, K}\), for all \(i \in I\).

\begin{lemma}\label{geom Z dense} Let \(V\) be an irreducible variety over \(K\)
and let us assume that \(\dcl_i(K) \subseteq \alg{K}\). Then, for any
\(\cL_i(K)\)-definable \(X\subseteq V\), \(X(M_i)\) is Zariski dense in \(V\) if
and only if it is \(K\)-Zariski dense in \(V\) (\emph{i.e.} is not contained in
any proper sub-variety of \(V\) over \(K\)).
\end{lemma}

\begin{proof}
Note first that this statement is clear for \(K = M_i\) --- \emph{i.e.} \(X(M_i)\) is Zariski dense in \(V\) if and only if it is \(M_i\)-Zariski dense in \(V\). Let us now assume that \(X(M_i)\) is \(K\)-Zariski dense in \(V\). We may assume \(M_i\) is sufficiently saturated and homogeneous. Let \(W\subseteq V\) be the Zariski closure of \(X(M_i)\). Then \(W\) is \(\aut(M_i/K)\)-invariant and hence it is defined over \(\dcl_i(K) \subseteq\alg{K}\). Let \((W_j)_{j\leq n}\) be the \(K\)-conjugates of \(W\). Then \(X(M_i) \subseteq \bigcup_i W_i \subseteq V\). So \(V = \bigcup_i W_i = W\).
\end{proof}

\begin{remark}
Since finite sets in fields can be coded using symmetric polynomials, we have that \(\acl_i(K) \subseteq \alg{\dcl_i(K)}\). It follows that \(\acl_i(K)\subseteq\alg{K}\) if and only if \(\dcl_i(K)\subseteq \alg{K}\).
\end{remark}

\begin{proposition}\label{equiv SPTC}
The following are equivalent:
\begin{enumerate}[(i)]
\item any regular totally \(T\) extension \(K\leq F\models T_I\) is \(\cL_I\)-existentially closed;

\item for every geometrically integral (affine) totally \(T\) variety \(V\) over \(K\), and, for all \(i\in I\), every \(p_i\in\TP^i(K)\) \(K\)-generic in \(V\), \(\bigcup_i p_i\) is realised in some \(K^\star \supsel K\);

\item for every geometrically integral (affine) totally \(T\) variety \(V\) over \(K\), and, for all \(i\in I_0\subseteq I\) finite, every \(\cL_i(K)\)-definable \(K\)-Zariski dense \(X_i(M_i)\subseteq V\), \(\bigcap_i X_i(K) \neq \emptyset\) --- equivalently, is Zariski dense in \(V\);

\item for all \(i\in I\), \(\dcl_i(K)\subseteq \alg{K}\) and for every geometrically integral (affine) totally \(T\) variety \(V\) over \(K\), and, for all \(i\in I'\subseteq I\) finite, every \(\cL_i(K)\)-definable Zariski dense \(X_i(M_i)\subseteq V\), \(\bigcap_i X_i(K) \neq \emptyset\) --- equivalently, is Zariski dense in \(V\);
\end{enumerate}
\end{proposition}

\begin{proof}
\begin{itemize}
\item[(i)\imp(ii)] Let \(p_i\) and \(V\) be as in (ii) and let \(a_i\models
p_i\) (in some \(M_i^\star \supsel M_i\)). Note that \(K\leq K(a_i)\) is
isomorphic to \(K(V)\). These isomorphisms allow us to make \(K(V)\) into a
model of \(T_I\) where all the \(a_i\) coincide. Since \(K\leq K(V)\) is regular
and totally \(T\), we find an \(\cL_I(K)\)-embedding \(f : K(V) \to K^\star\supsel
K\). The common image of the \(a_i\) in \(K^\star\) is a realisation of
\(\bigcup_i p_i\).

\item[(ii)\imp(iii)] Let \(X_i\) and \(V\) be as in (iii). If \(i\nin I_0\), let \(X_i(M_i) = V(M_i)\). Note that since \(M_i\) can be made into a model of \(T_K\) and \(V\) is totally \(T\), \(V(M_i)\) is Zariski dense in \(V\). For all \(i\in I\), by compactness, we find \(a_i\in X_i(M_i^\star)\), where \(M_i^\star \supsel M_i\) which is \(K\)-generic in \(V\). Applying (ii) to \(p_i = \tp_i(a_i/K)\), we find \(a\models \bigcup_i p_i\) in some \(K^\star \supsel K\); in particular, \(a \in \bigcap_{i\in I_0} X_i \sminus W\) where \(W\subseteq V\) is any \(K\)-subvariety.

\item[(iii)$\Leftrightarrow$(iv)] Let us first assume (iii) and prove that \(\dcl_i(K)\subseteq \alg{K}\). By contradiction, consider some \(a\in \acl_i(K)\sminus\alg{K}\). Let \(X_i\) be \(\cL_i(K)\)-definable such that \(X_i(M_i)\) is minimal finite containing \(a\). Then \(X_i(M_i)\) is \(K\)-Zariski dense in \(\Aa^1\). So \(X_i(K)\neq\emptyset\), contradicting the minimality of \(X_i\). The equivalence now follows from \cref{geom Z dense}.

\item[(iii)\imp(i)] Let \(F\) be as in (i) and \(X\) be a quantifier free \(\cL(K)\)-definable set with \(X(F)\neq\emptyset\). We may assume that \(X = \bigcap_{i\in I_0} X_i\), where \(I_0\subseteq I\) finite and \(X_i\) is \(\cL_i(K)\)-definable. Fix some \(a\in X(F)\). Let \(V\) be the locus of \(a\) over \(K\). Since \(K\leq F\) is regular and totally \(T\), \(V\) is geometrically integral and totally \(T\). By construction, \(X_i(M_i)\) is \(K\)-Zariski dense in \(V\). By (iv), \(X(K) = \bigcap_i X_i(K) \neq \emptyset\).
\qedhere
\end{itemize}
\end{proof}

\begin{definition} \label{PTCI}
The \(\cL\)-field \(K\) is said to be \emph{\(I\)-pseudo \(T\)-closed} (\(\SPTC\)) if the statements of \cref{equiv SPTC} hold.
\end{definition}

Conditions (i) and (ii) have a similar flavor to the interpolative fusion of \cite{KruTraChi}, although the main difference here is that \(K\) is not required to be a model of the \(T_i\) but only of \(T_{i,\forall}\).

\begin{example} The following are examples of \(\SPTC\) fields for different choices of \(T_i\) and \(T\):
\begin{itemize}
    \item By \cite[Theorem\ 1.7]{Pre-PRC}, $n$-\(\PRC\) fields are $\SPTC$, for \(T_i = \RCF_{<_i}\) and \(T = \bigvee_i T_i\), as in \cref{Ex PTC}.
    \item By \cite[Theorem\ 2.17]{MonRid-PpC}, $n$-\(\PpC\) fields are $\SPTC$, for \(T_i = \pCF_{v_i}\) and \(T = \bigvee_i T_i\), as in \cref{Ex PTC}.
    \item  More generally, it follows from \cite[Theorem\ 4.9]{Sch-RegT} that, if \(K\) is \(\PRC\) and \((<_i)_{i\in I}\) are distinct orders on \(K\), then it is \(\SPTC\), for \(T = \RCF\) and \(T_i = \RCF_{<_i}\). Similarly, for \(\PpC\) fields with named \(p\)-adic valuations.
    \item It follows from \cite[Theorem\ 2]{Kol-PAC}, that a \(\PAC\) field with one valuation is \(\SPTC\), for \(T = \ACF\) and \(T_0 = \ACVF\).  \item In the present paper we generalize this result by showing that a \(\PAC\) field with \(n\) distinct valuations is \(\SPTC\), for \(T = \ACF\) and \((T_i)_{i<n}\) copies of \(\ACVF\). This is a consequence \cref{free ec}.
    \item By \cite[Theorem\ 4.1]{Joh-MultVal}, if \(K\) is real closed (respectively $p$-adically closed or algebraically closed), and the \((v_i)_{0 < i \leq n}\) are distinct valuations, then it is \(\SPTC\), for \(T_0 = \RCF_{<_0}\) (respectively \(\pCF_{v_0}\) or \(\ACVF\)), \(T_i = \ACVF\), for \(0<i<n\) and \(T = \bigvee_i T_i\).
    \item As a consequence of \cref{free ec}, we also generalize this result by showing that if \(I = \{\Th(K_i)\mid i< n\}\) where \(K_i\) is either real closed or a characteristic zero Henselian valued field --- in an adequate language to eliminate quantifiers --- and \(T = \bigvee_i T_i\), then every \(\PTC\) field where the \(T_i\) induce distinct topologies is \(\SPTC\). Note that if all \(T_i\) but \(T_0\) are equal to \(\ACVF\), then any model of \(T_0\) is \(\PTC\).
\end{itemize}
\end{example}

\begin{lemma}
\label{FA Ti:reg}
Fix an \(i\in I\) and let \(K, M \models T_{i, \forall}\), \(F\leq M\) be
regular with \(\acl_i(F)\cap M\subseteq F\) and \(f : F\to K\) be an
\(\cL_i\)-embedding. Then \(K \tensor_F M\) can be made into a model of \(T_{i,
\forall}\), extending the \(\cL_i\)-structure on both \(K\) and \(M\).
\end{lemma}

\begin{proof}
By quantifier elimination, we can extend \(f\) to some \(g : M \to N \models T_i\). Since \(\acl_i(F)\cap M\subseteq F\), any \(\cL_i(F)\) definable set containing a tuple in \(M\setminus F\) is infinite and thus, in (a sufficiently saturated elementary extension of \(N\)), it contains a realization in \(N\setminus K\). By compactness, we may assume that \(K \aindep_F g(M)\). Since \(F\leq M\) is regular, we have \(K \lindep_F M\). Then \(K\tensor_F M \isom K g(M) \models T_{i, \forall}\).
\end{proof}

\begin{corollary}\label{better SPTC:reg}
Let \(K, M\models T_I\), \(F\leq M\) be regular and totally \(T\), with \(\acl_i(F)\cap M\subseteq F\), for all \(i\in I\), and \(f : F \to K\) be an \(\cL_I\)-embedding. Then \(K\leq K\tensor_F M\) is a regular totally \(T\) extension that can be made into a model of \(T_I\) extending the \(\cL_I\)-structure on \(K\) and \(M\).
\end{corollary}

\begin{proof}
This follows from \cref{better PTC,FA Ti:reg}.
\end{proof}

We will also need the following orthogonal case of free amalgamation:

\begin{lemma}\label{FA Ti:alg}
Fix an \(i\in I\) and let \(K, M \models T_{i, \forall}\), \(F\leq M\) be algebraic and \(f : F\to K\) be an \(\cL_i\)-embedding with \(M\tensor_F K\) integral. Then \(K \tensor_F M\) can be made into a model of \(T_{i, \forall}\), extending the \(\cL_i\)-structure on both \(K\) and \(M\).
\end{lemma}

\begin{proof}
By quantifier elimination we can extend \(f\) to some \(g : M \to N \models T_i\). Since \(F\leq M\) is algebraic and \(K \tensor_F M\) is integral, we have \(K\tensor_F M \isom K g(M) \models T_{i,\forall}\).
\end{proof}

Let us now prove that, in the cases which we will later consider, \(\SPTC\) is elementary:

\begin{proposition}\label{SPTC elem}
Assume that, for all \(i\in I\):
\begin{itemize}
\item[\textbf{(A$_i$)}] for every \(F\models T_{i,\forall}\), \(\acl_i(F) \subseteq \alg{F}\);
\item[\textbf{(Z$_i$)}] maximal Zariski dimension is definable in family: for
every \(\cL_i\) definable sets \(X\subseteq \Aa^{n+m}\), the set \(\{s
\in\Aa^{n}\mid \dim(X_s) = m\}\) is \(\cL_i\)-definable, where $X_s$ is the
fiber of $X$ in $s$.
\end{itemize}
Then the class of \(\SPTC\) fields is elementary --- in fact, inductive --- in \(\cL\cup\cL_I\).
\end{proposition}

\begin{remark}\label{better SPTC elem}
Instead of assuming Hypothesis \textbf{(A$_i$)}, we can allow the following generalization. For all \(i\in I\), let \(T_i'\supseteq T_{i,\forall}\) be such that for every \(F\models T_i'\), \(\dcl_i(F)\subseteq \alg{F}\). Then the class of \(\SPTC\) models of \(T_I' := \bigcup_i T_i'\) is elementary.

For example, one could take \(T_i'\) to be the class of \(\dcl_i\)-closed models of \(T_{i,\forall}\).
\end{remark}

\begin{proof}[Proof of \cref{SPTC elem}]
Let us first show that if \(F\leq K \models T_{I}\) is \(\cL_I\)-existentially closed and \(K\) is \(\SPTC\), then so is \(F\). Let \(F\leq M\models T_{I}\) be regular and totally \(T\). By Hypothesis \textbf{(A$_i$)}, for all \(i\), \(\acl_i(F)\cap M\subseteq \alg{F}\cap M = F\) and hence, by \cref{better SPTC:reg}, \(K\leq K\tensor_F M \models T_I\) is regular and totally \(T\) and the \(\cL_I\)-structure on \(K\tensor_F M\) extends both that of \(M\) and \(K\). Since \(K\) is \(\SPTC\), it follows that this extension is \(\cL_I\)-existentially closed and hence so is \(F\leq M\).

Now, we let $J$ be a set of indices and let \(K_j\) be \(\SPTC\), for all \(j\in J\), and \(\fU\) be an ultrafilter on \(J\). We wish to show that \(K := \prod_{j\to \fU} K_j\) is \(\SPTC\). Let \(V\) be a smooth geometrically integral totally \(T\) variety over \(K\) such that for all \(M\models T_K\), \(V(M)\neq\emptyset\). By definability of irreducibility in \(\ACF\), we can find smooth geometrically irreducible varieties \(V_j\) over \(K_j\) such that \(V = \prod_{j\to \fU} V_j\). Let \(Y := \{j\mid\) for all \(M_j\models T_{K_j}\),  \(V_j(M_j)\neq\emptyset\}\). If \(Y\nin\fU\), then for every \(j\nin Y\), let \(M_j\models T_{K_j}\) such that \(V_j(M_j) = \emptyset\). Then \(M:= \prod_{j\to\fU}M_j \models T_K\), but \(V(M) = \emptyset\), a contradiction. So, we may assume that \(Y = J\). For \(i\in I_0\subseteq I\) finite, let now \(M_i\models T_{i,K}\) and \(X_i(M_i) \subseteq V\) \(\cL_i(K)\)-definable and Zariski dense --- in other words \(\dim(X_i) = \dim(V)\) or equivalently, any coordinate projection which is dominant from \(V\) is dominant from \(X_i\). By Hypothesis \textbf{(Z$_i$)}, we find \(M_{ji}\models T_{i,K_j}\) and \(X_{ji}(M_{ji})\subseteq V_j\) \(\cL_i(K)\)-definable and Zariski dense with \(\prod_{j\to \fU} X_{ji} = X_i\). Since \(K_j\) is \(\SPTC\), \(\bigcap_i X_{ji}(K_j) \neq \emptyset\). Hence, \(\bigcap_i X_{i}(K) \neq \emptyset\).

It follows that the class \(\SPTC\) is elementary in \(\cL\cup\cL_I\). Note that, since it is in fact closed under existentially closed substructures, it is, in fact, inductive this can be seen, \emph{e.g.} by working in the enrichment of \(\cL\cup \cL_I\) by all existential formulas which yields a class closed under all substructures, \emph{i.e.} a universal class.
\end{proof}

\section{V-Topological theories}\label{VTopo}

We now fix an \(i \in I\) and \(M_i\models T_i\). We now wish to consider that \(\cL_i\)-definable sets are essentially open. This is closely related to the notion of "t-theory" in \cite[Chapter\ III]{vdD-PhD}.

\begin{hyp}\label{H2i}
In this section, we will assume that:

\begin{itemize}
\item[\textbf{(H$_{i}$)}] 
\(T_i\) admits a definable (non-discrete) Henselian \(V\)-topology \(\tau_i\) and, for every
 \(M_i\models T_i\), any \(\cL_i(M_i)\)-definable set \(X_i\) has non-empty
 \(\tau_i\)-interior in (the \(M_i\)-points of) its Zariski closure.
\end{itemize}
\end{hyp}

\begin{example}
\begin{enumerate}
\item Hypothesis \textbf{(H$_{i}$)} holds if \(T_i\) is the theory of real
closed fields. Indeed, the order topology is a definable Henselian
\(V\)-topology and any definable set is a disjoint union of sets of the form
\(V\cap U\) where \(V\) is Zariski closed and \(U\) is open.
\item Similarly, Hypothesis \textbf{(H$_{i}$)} also holds if \(T_i\) is a theory of
(\(\RV_1\)-enriched) Henselian valued fields which eliminates quantifiers
relatively to \(\RV_1\) --- or, more generally, relatively to \(\bigcup_n\RV_n\) in
characteristic zero.
\end{enumerate}
\end{example}

\begin{remark}\label{rem elem Hi}
\begin{enumerate}
\item Whenever \textbf{(H$_{i}$)} holds, Hypothesis \textbf{(Z$_i$)} of
\cref{SPTC elem} also holds since a definable set is Zariski dense in \(\Aa^n\)
if and only if it has non-empty \(\tau_i\)-interior.
\item As we will see in \cref{acl Ti}, Hypothesis \textbf{(A$_i$)} of \cref{SPTC
elem} also holds provided \(\dcl_i(\emptyset) \subseteq \alg{F}\) where \(F\) is
the field generated by the constants --- constants that we may add for this
exact purpose.
\end{enumerate}
\end{remark}

As it turns out, we can almost control the parameters in Hypothesis \textbf{(H$_i$)}:

\begin{lemma}\label{acl Ti}
Let \(F\leq M_i\), the following are equivalent:
\begin{enumerate}[(i)]
\item \(\dcl_i(\emptyset)\subseteq\alg{F}\);
\item \(\acl_i(F)\subseteq \alg{F}\);
\item any \(\cL_i(F)\)-definable set \(X\) has non-empty \(\tau_i\)-interior in its \(F\)-Zariski closure.
\end{enumerate}
\end{lemma}

\begin{proof} It is clear that (ii) implies (i).
\begin{itemize}
\item[(i)\imp (ii)] Let \(E = \dcl_i(\emptyset)\). We prove, by induction on
\(\dim(V)\), that for any variety \(V\) over \(E\) and \(X\subseteq V\times
\Aa^1\) \(\cL_i\)-definable with \(X_a := \{y\in \Aa^1 \mid (a,y)\in X\}\)
finite for every \(a\in X(M_i)\), we have \(X_a(M_i) \subseteq \alg{E(a)}\). We
may assume that \(V\) is irreducible. By finiteness of \(X_a\), \(X\) is not
\(\tau_i\)-open in \(V\times \Aa^1\) and hence is not Zariski dense. So there
exists \(P = \sum_j p_j(x) y^j \in E[x,y]\) with \(P(V\times \Aa^1)\neq 0\) but
\(P(X)=0\). Let \(V_j\) be the zero locus of $p_j$ in \(V\) and \(J := \{j\mid
V_j\subset V\} \neq \emptyset\). By induction, for any \(j\in J\) and \(a\in
V_j\), \(X_a(M_i)\subseteq \alg{E(a)}\) and by construction, for any \(a\in
V\sminus \bigcup_j V_j\), \(P(a,y) \neq 0\) and \(P(a,X_a) = 0\).
\item[(ii)\imp(iii)] By \cref{geom Z dense}, \(X(M_i)\) is Zariski
dense in its \(F\)-Zariski closure \(V\). It now follows from \textbf{(H$_i$)}
that \(X\) has non-empty interior in \(V\).

\item[(iii)\imp(ii)] Let \(X\subseteq\Aa^1\) be a minimal finite \(\cL_i(F)\)-definable set. Let \(V\) its \(F\)-Zariski closure. Then either \(\dim(V) = 0\), in which case \(X\subseteq \alg{F}\) as required, or \(\dim(V) = 1\) and \(V = \Aa^1\). By (iii), the finite set \(X\) is \(\tau_i\)-open in \(\Aa_1\), contradicting its non-discreteness.\qedhere
\end{itemize}
\end{proof}

\begin{lemma}\label{induced V-top} Let \(F\leq M_i\) be such that, for any
 non-empty open \(\cL_i(F)\)-definable subset of \(\Aa^1(M_i)\) has an
 \(F\)-point. Then the topology generated by the \(X(F)\), where \(X\subseteq
 M_i\) is open and \(\cL_i(F)\)-definable, is a \(V\)-topology. If, moreover,
 \(\sep{F}\cap M_i\subseteq F\), then it is Henselian.
\end{lemma}

We will also denote this topology on \(F\) by \(\tau_i\) even if it might not be the induced topology, per se.

\begin{proof}
Let \(S\) be \(\cL_i(F)\)-definable and bounded (in \(M_i\)) and let \(U\) be \(\cL_i(F)\)-definable neighborhood of zero. Then \(\{x\in M_i\mid xS\subseteq U\}\) is a \(\cL_i(F)\)-definable neighborhood of \(0\). So, by hypothesis, its \(\tau_i\) interior contains a point distinct from \(0\); in other terms, \(S(F)\) is bounded in \(F\). So, if \(U\subseteq M_i\) is an \(\cL_i(F)\)-definable neighborhood of \(0\), \((M_i\sminus U)^{-1}\) is bounded in \(M_i\) and hence in \(F\), proving that the topology is a \(V\)-topology.

The second statement follows from the fact that Henselianity states the existence of roots to certain separable polynomials.
\end{proof}

\begin{lemma}\label{open Z dense}
Let \(F\leq M_i\models T_i\) be such that \(\sep{F}\cap M_i\subseteq F\) and any non-empty \(\tau_i\)-open \(\cL_i(F)\)-definable subset of \(\Aa^1(M_i)\) has an \(F\)-point. Let \(V\) be a geometrically irreducible variety over \(F\) and \(X(M_i)\subseteq V(M_i)\) be \(\tau_i\)-open and \(\cL_i(F)\)-definable. The following are equivalent:
\begin{enumerate}[(i)]
\item \(X(M_i)\) is Zariski dense in \(V\);
\item \(X(M_i)\cap \sm{V} \neq \emptyset\);
\item \(X(F)\cap \sm{V} \neq \emptyset\).
\end{enumerate}
\end{lemma}

\begin{proof}
Note that (i) obviously implies (ii), and so does (iii). Also, shrinking \(V\), we may assume (\cref{smoothpoint}) that it is affine smooth and that there exists an étale map \(f : V \to \Aa^n\) over \(F\). 
\begin{itemize}
\item[(ii)\imp(iii)] By  \cref{InvFuncThe}, \(f(X)\subseteq \Aa^n\) has non-empty interior, so there exists \(y\in f(X)\cap \Aa^n(F)\) and hence  \(x\in X(M_i)\) with \(f(x) = y\). Then \(x\in \sep{F(f(x))}\cap M_i = F\).

\item[(ii)\imp(i)]
By  \cref{InvFuncThe} we can find a non-empty \(\cL_i(M_i)\)-definable \(\tau_i\)-open \(U\subseteq X(M_i)\) such that \(\restr{f}{U}\) is a homeomorphism onto its open image. Since open subset of \(\Aa^n\) are Zariski dense, (i) follows.\qedhere
\end{itemize}
\end{proof}

\begin{lemma}\label{Z dense Ti}
Let \(F\leq M_i\), with \(\dcl_i(\emptyset)\subseteq\alg{F}\), let $V$ be a geometrically irreducible variety over $F$ and \(X\subseteq V\) be \(\cL_i(F)\)-definable. The following are equivalent:
\begin{enumerate}[(i)]
\item \(X(M_i)\) is Zariski dense in \(V\);
\item \(X(M_i)\) is \(F\)-Zariski dense in \(V\);
\item the \(\tau_i\)-interior of \(X(M_i)\) in \(V\) contains a smooth point.
\end{enumerate}
\end{lemma}

\begin{proof} It is clear that (i) implies (ii) and, by \cref{open Z dense}, (iii) implies (i). So let us prove that (ii) implies (iii) and assume that \(X(M_i)\) is \(F\)-Zariski dense in \(V\) and let \(U\) be its \(\tau_i\)-interior in \(V\). Since  \(X\sminus U\) has empty \(\tau_i\)-interior in \(V\), by \cref{acl Ti}, it is not \(F\)-Zariski dense in \(V\). It follows that \(U\) is \(F\)-Zariski dense in \(V\). In particular \(U\cap \sm{V} \neq \emptyset\).
\end{proof}

\begin{remark}
Note that Hypothesis \textbf{(Z$_i$)} of \cref{SPTC elem} follows from
\textbf{(H$_{i}$)} (see \cref{H2i}). Indeed, by \cref{Z dense Ti}, a
definable set \(X\subseteq \Aa^n\) is Zariski dense if and only if it has non empty \(\tau_i\)-interior.

Also, if \(F\models T_{i,\forall}\) is such that \(\dcl_i(\emptyset) \subseteq \alg{F}\), then, by \cref{acl Ti}, \(T_{i, F, \forall}\) satisfies the conditions of \cref{better SPTC elem}. Therefore, for topological theories, the conditions for \(\SPTC\) to be elementary given in \cref{SPTC elem,better SPTC elem} are easy to ensure.
\end{remark}

We will now show the remarkable fact that, under some finiteness hypothesis,  \(\cL_I\)-existential closure comes essentially from independence of the \(\tau_i\) topologies.

\begin{hyp} We will consider the following hypothesis:
\begin{itemize} \label{F2}
\item[(\textbf{F}$_2$)] The set \(I\) is finite and for every \(M\models T_K\), there exists \(i\in I\) and an \(\Lrg(K)\)-embedding \(M_i \to M^\star \supsel M\).
\end{itemize}
\end{hyp}
In particular, the minimal fields of the form \(M\cap\sep{K}\), where \(M\models T_K\) are among the \(K_i = M_i\cap\sep{K}\), where \(M_i\models T_{i,K}\). Note that it also follows that every extension \(K\leq M\models T_I\) is totally \(T\).

It follows that a field is then \(I\)-pseudo \(T\)-closed if and only if it is \(I\)-pseudo \(\bigvee_{i\in I} T_i\)-closed, so the mention of \(T\) is now redundant.

\begin{definition}\label{PIC}
If \(I\) is finite, we will say that \(K\) is pseudo \(I\)-closed (\PIC) if it is \(I\)-pseudo \(\bigvee_{i\in I} T_i\)-closed.
\end{definition}

\begin{theorem}\label{free ec}

Assume \textbf{(H$_{i}$)} for all \(i\in I\) (see \ref{H2i}) and (\textbf{F}$_2$).
The following are equivalent:
\begin{enumerate}[(i)]
\item \(K\models\PIC\);
\item the following hold:
\begin{enumerate}
	\item \(K\models \PTC\);
	\item for all \(i \in I\), \(\dcl_i(\emptyset)\subseteq \alg{K}\);
	\item the \(\tau_i\) are distinct topologies on \(K\) as \(i\in I\) varies; 
	\item for all \(i \in I\), any non-empty open \(\cL_i(K)\)-definable \(X\subseteq \Aa^1\) contains a \(K\)-point.
\end{enumerate}
\end{enumerate}
\end{theorem}

\begin{proof}
Assuming (i), (a) obviously follows. Condition (b) follows from \cref{equiv
SPTC}.(iv) and conditions (c) and (d) follow from \cref{equiv SPTC}.(iii)
applied to \(\Aa^1\). Conversely, let us assume (ii). By \cref{equiv SPTC,Z
dense Ti,open Z dense}, it suffices to prove that, given \(V\) a geometrically
integral totally \(T\) variety over \(K\) and \(X_i\subseteq V\)
\(\cL_i(K)\)-definable \(\tau_i\)-open with \(X_i(K_i)\cap\sm{V}\neq\emptyset\),
we have \(\bigcap_i X_i(K)\neq\emptyset\). 

Note that, by Hypothesis \textbf{(F$_2$)}, any \(K_M = M\cap \sep{K}\), where \(M\models T_K\), is an algebraic extension of some \(K_i\), which, by \cref{induced V-top} admits a Henselian V-topology extending \(\tau_i\). So \textbf{(H)} holds and if \(K_M\) is minimal in some \(C_\tau\) and not separably closed, then \(\tau =\tau_i\) and \(K_M\isom K_i\) --- in particular, \textbf{(F$_1$)} holds. Since the \(\tau_i\) are all distinct, it also follows that every \(K_i\) is minimal in \(C_{\tau_i}\). We can, therefore, conclude with \cref{approx var}.
\end{proof}

\begin{remark}
If \textbf{(F$_2$)} does not hold but, instead, \textbf{(H)} holds and no
\(M\models T_K\) contains \(\sep{K}\), then, by \cite[Theorem\ 4.9]{Sch-RegT},
\cref{free ec} also holds.
\end{remark}

\section{Bounded \texorpdfstring{\(\PIC\)}{PIC} fields} 
\label{BoundedPTC}

From now on, we assume that $I$ is finite and that \(T = \bigvee_{i\in I} T_i\).
In particular, any extension \(K\leq M\models T_I\) is totally \(T\), and we
will therefore prefer the terminology pseudo \(I\)-closed (see \cref{PIC}).

We fix some \(I_\tau \subseteq I\), and we assume that, for every \(i\in
I_\tau\), \textbf{(H$_i$)} holds (see Hypothesis \ref{H2i}). By \cref{rem elem Hi},
hypotheses \textbf{(Z$_i$)} and \textbf{(A$_i$)} of \cref{SPTC elem} holds for
every \(i\in\tau\) (provided we add constants). We also assume that
\textbf{(Z$_i$)} holds for every \(i\nin I_\tau\), and we fix \(\Td_i\) as in
\cref{better SPTC elem}.

\begin{definition}
A field $K$ is called \emph{bounded} if for any integer $n$, $K$ has finitely many extensions of degree $n$.
\end{definition}

\begin{notation}\label{bSPTC}
Fix \(\fd : \Zz_{>0} \to \Zz_{>0}\). Let \(\cL_\fd := \cL_I
\cup\{c_j \mid j > 0\) and \(|c_j| = \fd(j)\}\). We denote by \(T_\fd\) the
\(\Ld\)-theory of fields such that \(P_i := X^{\fd(i)} + \sum_{j<\fd(i)} c_{i,j}
X^j\) is irreducible and every separable polynomial of degree \(i\) is split
modulo \(P_i\). Let \(\bSPTC\) be the (elementary) class of \(\PIC\) models of
\(T_\fd\cup\bigcup_{i\nin I_{\tau}}\Td_i\).

From now on, we will always work in the language $\Ld$, so if $K$ is $\bSPTC$ and
$A\leq K$, then $\acl(A), \dcl(A), \acl_i(A), \dcl_i(A)$ always contain the
constants.
\end{notation}

\begin{remark}
Any model of \(T_\fd\) is bounded. Conversely, any bounded model of \(T\) can be made into a model of \(T_\fd\) for some well-chosen \(\fd\).
\end{remark}

\begin{lemma}\label{char gal}
Let \(K\models T\) and \(F \leq K\). The following are equivalent:
\begin{enumerate}
\item \(K\models T_\fd\) and \(\res : \Gal(K)\to\Gal(F)\) is surjective;
\item \(F\models T_\fd\) and \(\res : \Gal(K)\to\Gal(F)\) is a homeomorphism;
\item \(K\models T_\fd\) and \(F\models T_\fd\).
\end{enumerate}
\end{lemma}

This is rather standard, but we could not find a reference, so we sketch the proof.

\begin{proof}
Let us first assume (1) and prove (3). The polynomial \(P_n\) is clearly
irreducible over \(F\). Let \(Q \in F[X]\) be an irreducible separable
polynomial of degree \(n\) and $a$ one of its roots. Since $Q$ is split in
$K[X]/P_n$, $P_n$ admits a root $e\in \sep{F}$ of $P_n$ such that $a \in K[e]$.
Then $\Gal(F[e]) = \res(\Gal(K[e])) \leq \res(\Gal(K[a])) = \Gal(F[a])$ where the
equalities follow from surjectivity; and hence $F[a]\leq F[e]$.

If (3) holds, then every separable extension of $K$ is included by one generated
by an element of $\sep{F}$ and hence \(\res : \Gal(K)\to\Gal(F)\) is injective.
Also, any separable extension of $F$, is contained in one generated by a root of
a \(P_n\) which is also irreducible over $K$. It follows that $\sep{F}$ is
linearly disjoint from \(K\) over \(F\) and hence \(\res : \Gal(K)\to\Gal(F)\)
is surjective, so (2) holds. Conversely, if (2) holds, then, by Galois
correspondence, the $P_n$ are irreducible over $K$ and $\sep{K} = K\sep{F}$ so
$K\models T_\fd$ and (1) holds.
\end{proof}

\begin{lemma}\label{forth}
Let \(K,M\models\bSPTC\), with \(K\) perfect, and \(F \leq K\) with
\(\acl_i(F)\cap K \subseteq F\), for all \(i\in I\). Then any
\(\cL_\fd\)-embedding \(f : F \to M\) can be extended to an
\(\cL_\fd\)-embedding \(g:K\to M^\star\supsel M\) with \(\acl_i(g(K))\cap
L^\star\subseteq g(K)\), for all \(i\in I\).
\end{lemma}

\begin{proof}
Note that \(F\) is perfect and hence \(F\leq K\) is regular. So \(g\) exists by \cref{better SPTC:reg} and the fact that \(M\models\PIC\). There remains to show that \(\acl_i(g(K))\cap M^\star\subseteq g(K)\). Since \(\acl_i(K)\subseteq \alg{K}\), it suffices to show that \(\alg{g(K)}\cap M^\star\subseteq g(K)\). By \cref{char gal}, since \(K\models T_\fd\) and \(F\leq K\) is regular, \(F\models T_\fd\) and hence, by \cref{char gal} again, \(\Gal(M^\star) \to \Gal(g(F))\) is a homeomorphism. It factorizes through the homeomorphism \(\res : \Gal(K) \to \Gal(F)\) and hence \(\res:\Gal(M^\star)\to\Gal(g(K))\) is surjective; \emph{i.e.} \(\alg{g(K)}\cap M^\star = g(K)\).
\end{proof}

\begin{definition}
For every \(K\models\bSPTC\) and \(A\subseteq K\), let \(\acl_I(A) \subseteq K\) be the smallest subset \(C\subseteq K\) containing \(A\) such that \(\acl_i(C)\subseteq C\), for all \(i\in I\).
\end{definition}

\begin{proposition}\label{char elem}
Let \(K,M\models \bSPTC\) be perfect, \(F \leq K\) and \(f : F\to M\) be an \(\Ld\)-embedding. The following are equivalent:
\begin{enumerate}
\item \(f\) extends to an \(\Ld\)-embedding \(g : \acl_I(F)\cap K \to M\);
\item \(f\) is \(\Ld\)-elementary (from \(K\) to \(M\)).
\end{enumerate}
\end{proposition}

\begin{proof}\ 
\begin{itemize}
\item[(1) $\Rightarrow$ (2)] Let \(g : \acl_I(F)\cap K \to M\) extend \(f\). By
a standard back and forth argument using \cref{forth}, we build \(K\subsel
K^\star\), \(M\subsel M^\star\) and an \(\Ld\)-isomorphism \(h : K^\star \to
M^\star\) extending \(g\). Elementarity of \(f\) follows immediately.
\item[(2) $\Rightarrow$ (1)] This is immediate.\qedhere
\end{itemize}
\end{proof}

\begin{corollary}\label{mod compl}
Any \(K\leq M\), with both fields \(\bSPTC\) perfect, is \(\Ld\)-elementary.
\end{corollary}

\begin{proof}
By  \cref{char gal}, the morphism \(\res:\Gal(K)\to\Gal(F)\) is surjective. Since \(\acl_i(K)\subseteq \alg{K}\), it follows that \(\acl_I(K) \cap M \subseteq K\) and hence, by \cref{char elem} the inclusion is elementary.
\end{proof}

For the rest of the section we fix a perfect \(K\models \bSPTC\) and \(M_i\models T_{i,K}\).

\begin{proposition}\label{char acl}
 Let $F \leq K$. Then $\acl(F) = \acl_I(F) \cap K$.
\end{proposition}

\begin{proof}
Let $E = \acl_I(F) \cap K\subseteq \acl(F)$. By \cref{better SPTC:reg}, we find an \(\Ld(E)\)-embedding \(f : K \to K^\star \supsel K\) with \(f(K) \lindep_E K\). By \cref{mod compl}, \(f\) is elementary. It follows that \(\acl(F) \subseteq K\cap f(K)= E\).
\end{proof}

\begin{notation}
For every \(p\in \TP(F)\), where \(F\leq K\), let \(p_i\) denote the underlying
(quantifier free) \(\cL_{i}\)-type.
\end{notation}

\begin{proposition}\label{loc dense}
Let \(F = \acl(F) \leq K\),  $V$ an irreducible variety over $F$, \(p\in\TP(F)\)
\(F\)-generic in \(V\) and, for \(i\in I_\tau\), let \(O_i\subseteq V(M_i)\) be
\(\cL_i(K)\)-definable \(\tau_i\)-open sets consistent with \(p_i\) (modulo
\(T_{i, K}\)). Then \(p\) is consistent with \(\bigcap_{i\in I_\tau} O_i(K)\).
\end{proposition}

\begin{proof}
Let \(a\models p\) (in some \(K^\star\supsel K\)). Let \(ac\) enumerate
\(\acl(Fa)\). By \cref{char elem}, it suffices to find \(de \qfequiv_F ac\) with
\(d\in \bigcap_{i\in I_\tau} O_i\). By compactness, we may assume \(c\) is
finite. Let \(q_i = \tp_i(ac/F)\), for every \(i\in I\), and \(W\) be the
geometrically integral algebraic locus of \(ac\). For every \(i\in I_\tau\), we
also denote by \(O_i\) the corresponding open subset of \(W\).

Let us fix some \(i\in I_\tau\). For every \(\cL_i(F)\)-definable \(X\), by
\textbf{(H$_i$)}, \(q_i\vdash X\) if and only if \(q_i\vdash \mathring{X}\cap
\sm{W}\), where \(\mathring{X}\) denotes the \(\tau_i\)-interior of \(X \cap W\)
in \(W\). In particular, \(\mathring{X}\cap \sm{W}\cap O_i \neq \emptyset\) and,
by \cref{open Z dense}, \(X\cap O_i\) is consistent with the \(K\)-generic of
\(W\). So, by compactness, we find \(a_ic_i \models q_i\) \(K\)-generic in \(W\)
with \(a_i \in O_i\). If \(i\nin I_\tau\), by \cref{FA Ti:reg} we can also find
\(a_ic_i\models q_i\) \(K\)-generic in \(W\). Let \(r_i := \tp_i(a_ic_i/K)
\supseteq q_i\). Note that since \(\tp(ac/F)\) is finitely satisfiable in \(K\),
\(\sm{W}(K)\neq \emptyset\). By \cref{equiv SPTC}, there exists \(de \models
\bigcup_i r_i\) in some \(K^\star\supsel K\). By construction, for every \(i\in
I_\tau\), we have \(d\in O_i\) and, for every \(i\in I\), we have \(de\models
p_i\), concluding the proof.
\end{proof}

Let \(\tau\) be the topology generated by the \(\tau_i\), for \(i\in I_\tau\).

\begin{theorem}\label{Density}
Let \(F = \acl(F)\leq K\), and \(X\) be \(\Ld(F)\)-definable and \(\tau\)-closed (respectively \(\tau\)-open) in its Zariski closure \(V\). Then \(X\) is quantifier free \(\bigcup_{i\in I_\tau} \cL_i(F)\)-definable.
\end{theorem}

\begin{proof}
Let assume that \(X\) is \(\tau\)-closed. Let \(x\in X(K)\) and \(y\in V(K)\) be
such that \(\bigcup_{i\in I_\tau} \tp_i(x/F) = \bigcup_{i\in I_\tau}
\tp_i(y/F)\). By compactness, it suffices to show that \(y\in X\). For every
\(i\in I_\tau\), \(O_i \subseteq V(M_i)\) \(\tau_i\)-open containing \(y\). Let
\(W\) be the locus of \(x\) (and hence \(y\)) over \(F\). By \cref{loc dense},
\(\tp(x/F)\) is consistent with \(W\cap \bigcap_i O_i\), In particular,
there exists \(z \in X\cap \bigcap_i O_i\). So \(y\) is in the \(\tau\)-closure
of \(X\), which is \(X\) itself.
\end{proof}

Considering the \(\tau\)-closure of a definable set, we have generalized the
"density theorem" of \cite[Theorem 3.17]{Mon-NTP2}:

\begin{corollary}\label{Density2}
Let \(F = \acl(F)\leq K\), and \(X\) be
\(\Ld(F)\)-definable with Zariski closure \(V\). Then there exists an integer
\(m\), \(\cL_i(F)\)-definable $\tau_i$-open sets \(Y_{ij}\subseteq V\), for
\(i\in I_\tau\) and \(j<m\), and a quantifier free $\Ld(F)$-definable set $Y$,
with \(\dim(Y)< \dim(X)\), such that
\[X \subseteq \bigcup_j \bigcap_i Y_{ij} \cup Y \text{ is \(\tau\)-dense.}\]
\end{corollary}

\begin{proof}
Applying \cref{Density} to the $\tau$-closure of $X$, we find
\(\cL_i(F)\)-definable sets \(Y_{ij}\) such that \(X\subseteq \bigcup_j
\bigcap_i Y_{ij}\) is \(\tau\)-dense. Further dividing, we may assume that
\(Y_i\) is either \(\tau_i\)-open or has empty \(\tau_i\)-interior (and hence
has lower dimension by \textbf{(H$_i$)}). 
\end{proof}

\begin{remark}
The previous result generalizes the common pattern, that, when adding a
"generic" topology to some structure, the "open core" will consist only of sets
definable in the topology on its own. In fact, this result provides a broad
setting in which to consider these questions, which does not depend on the
initial structure and allows for more than one topology.
\end{remark}

\subsection*{Generic derivations}
To illustrate that point, let us consider the following example: let \(T_1 =
\DCF{0}\) the (Morleyized) theory of characteristic zero differentially closed
fields with a derivation, \(T_2 = \ACVF\), \(I := \{T_1,T_2\}\). Then
existentially closed valued fields with a derivation are \(\PIC\). Hence,
definable closed sets in such fields are definable using only the valued field
structure, recovering a special case of \cite[Theorem~3.1.11]{CKPoi-GenDer}.

The proof given here would extend immediately to bounded perfect \(\PAC\) fields
with independent valuations and generic commutative derivations (which are
\(\PIC\) for \(T_1 = \DCF{n}\) and \(T_2,\ldots T_m = \ACVF\)), if we knew that
this class is indeed closed under elementary extensions. \Cref{SPTC elem} does
not apply in that context, but a geometric axiomatization seems plausible, along
the lines of \cite{Tre-UC}.

There is another issue if we wish to consider fields that are not \(\PAC\),
which is our reliance on relational languages. If \(T_2 = \RCF_<\) and \(T_1 =
\DCF{0}\) (Morleyized), then existentially closed ordered differential fields
might not be \(\PIC\). Indeed, the order on a subfield of a differential field
might not extend to the whole field. However, the setup presented here can be
adapted to allow for non-relational \(\cL_i\); the main issue being, again, the
elementarity of the class.

\section{Amalgamation and Burden in \texorpdfstring{\(\bSPTC\)}{PTCId}} \label{Burden}

Recall that we now assume that \(I\) is finite and that \(T = \bigvee_{i\in I} T_i\).
In particular, any extension \(K\leq M\models T_I\) is totally \(T\). We now
work with \(I_\tau = I\) and therefore assume that \textbf{(H$_i$)} holds for
every \(i\in I\) (see Hypothesis \ref{H2i}) and \(\tau\) will denote the
topology generated by the \(\tau_i\) for all \(i\in I\). Finally, we fix
\(K\models\bSPTC\) perfect sufficiently saturated and homogeneous (see
\cref{bSPTC}) and \(M_i\models T_{i,K}\).

We start by proving that obstructions to 3-amalgamation only arise from the quantifier free structure.

\begin{thm}\label{Amalgamation}
Let \(F\leq K\) and $a_1, a_2, c_1, c_2, c\in K$. Suppose that \(\dcl_i(\emptyset)\subseteq \alg{F}\), for all \(i\), \(F = \acl(F)\), \(a_j\) enumerates \(A_j := \acl(Fa_j)\) and \(c_j\) enumerates \(C_j := \acl(Fc_j)\), for \(j=1,2\). Assume also that $A_1 \cap A_2 = F$, $c \aindep_{F} a_1 a_2$, $c_1 \equiv_{F} c_2$ and $c \qfequiv_{A_j} c_j$, for \(j = 1,2\). Then \[\tp(c_1/A_1) \cup \tp(c_2/A_2) \cup \qftp(c/A_1A_2)\] is consistent.
\end{thm}

This is a minor variation on \cite[Theorem\,3.21]{Mon-NTP2}. 

\begin{proof}
First note that, since $A_1A_2$ is a regular extension of $A_1$ and of $A_2$, by \cite[Lemma\ 2.1]{Cha-OmPAC}, we have that $\alg{A_1} \cap \alg{A_2} = \alg{(A_1\cap A_2)} = \alg{F}$. Note also that for any \(F\leq E\leq K\), by \cref{acl Ti}, \(\acl(E) = \alg{E}\cap K\).

Let \(C := F(c)\) and \(A := \acl(A_1A_2)\). Since $c \qfequiv_{a_j} c_j$, we have an \(\cL_I(A_j)\)-isomorphism \(f_j : A_jC_j \to A_jC\) sending \(c_j\) to \(c\). This map extends to an \(\Lrg(A_j)\)-embedding \(B_j := \acl(A_jC_j) \to \alg{(A_jC)}\). We endow \(D_j := f_j(B_j)\) with the \(\cL_I\)-structure making \(f_j\) into an \(\cL_I(A_j)\)-isomorphism. This structure extends the \(\cL_I\)-structure on \(A_jC\).

Note that both \(A_2\leq A\) and \(A_2\leq D_2 \isom B_2\) are regular and that,
since \(a_1 \aindep_{a_2} c\), we have \(A\aindep_{A_2} D_2\). It follows that
\(A\leq D_2A\), \(A_2 \leq D_2A\) and \(F\leq D_2A\) are regular. So \(D_2A
\lindep_{F} \alg{F}\) and hence \(D_2A \lindep_{CA_1} \alg{F}CA_1\). Moreover,
since \(\alg{C}\lindep_{\alg{F}}\alg{A_1}\alg{A_2}\), by
\cite[lemma\,2.5.(2)]{Cha-OmPAC} we have:
\[\alg{(CA_1)}\cap\alg{(CA_2)}\alg{(A_1A_2)} =
\alg{(\alg{C}(\alg{A_1}\cap\alg{A_2}))}\alg{A_1} = \alg{C}\alg{A_1}.\] By
\cref{char gal}, \(\Gal(C)\simeq\Gal(F)\) and hence, \(\alg{C} = \alg{F}C\).
Similarly, \(\alg{A_1} = \alg{F}A_1\) and thus \(\alg{C}\alg{A_1} =
\alg{F}CA_1\). So, \(\alg{(CA_1)} \cap D_2A = \alg{F}CA_1 \cap D_2A = CA_1\);
\emph{i.e.} \(CA_1\leq D_2A\) is regular. By symmetry, \(CA_2\leq D_1A\) is also
regular.

Since \(D_j \lindep_{CA_j} CA \leq D_kA\), for \(k\neq j\), by \cref{FA Ti:reg}, \(D_jA\) can be made into a model of \(T\) whose \(\cL_I\)-structure extends that of \(D_j\) and \(CA\). Since \(D_1\lindep_{CA_1} D_2 A\), we also have \(D_1A\lindep_{CA} D_2 A\) and, by \cref{FA Ti:alg}, \(D_1D_2A\) can be made into a model of \(T\) whose \(\cL_I\)-structure extends that of \(CA\), \(D_1\) and \(D_2\). Also, since \(CA_1\leq D_2A\) is regular and \(D_1\leq \alg{(CA_1)}\), it follows that \(D_1 \leq D_1D_2A\), and thus \(F\leq D_1D_2A\), are regular. So \(D_1D_2A\lindep_A \alg{F}A = \alg{A}\) and \(A\leq D_1D_2A\) is regular.

By \cref{better SPTC:reg}, there exists an \(\cL_I(A)\)-embedding \(g : D_1D_2A \to K^\star\supsel K\). Let \(c^\star = g(c)\). We have \(c^\star\qfequiv_{\cL_I(a_1a_2)} c\). Moreover, by \cref{char elem}, \(g\comp f_j\) is a \(\cL_I(A_j)\)-elementary isomorphism sending \(c_j\) to \(c^\star\).
\end{proof}

\subsection{Burden}

We will now study the burden of perfect bounded $\PIC$ fields in terms of the
burden of each theory $T_{i}$.

\begin{notation}
We denote by $\bdn_i$ the burden computed in models of $T_{i,K}$.
\end{notation}

\begin{lemma}\label{bdni}
Let  \(X\) be an \(\cL_i(M_i)\)-definable set of dimension \(d\). Then
\[\bdn_i(X)=\bdn_i(\Aa^d).\]
\end{lemma}

\begin{proof}
By \cref{bounded cover}, there exists an \(\cL_i(M_i)\)-definable bounded
$\tau_i$-open neighborhood \(O_i\) of \(0\) such that \(M_i \subseteq O_i\cup
(O_i\sminus\{0\})^{-1}\). Then \(\Aa^d(M_i) = \bigcup_{\epsilon \in 2^d}
\prod_{j<d} O_i^{\epsilon_j}\) where \(O_i^1 = O_i\) and \(O_i^0 =
(O_i\sminus\{0\})^{-1}\). Since \(\ \bdn_i(\prod_{j<d} O_i) = \prod_{j<d}
O_i^{\epsilon_j}\) for every \(\epsilon\), it follows that \(\bdn_i(\prod_{j<d}
O_i) = \bdn_i(\Aa^d)\). If \(U\subseteq \Aa^d\) is an \(\cL_i(M_i)\)-definable
open, it contains a definable subset of the form \(\prod_{j<d} U_j\) where \(U_j
\subseteq M_i\) is in (affine) \(\cL_i(M_i)\)-definable bijection with \(O_i\).
It follows that \[\bdn_i(\Aa^d) = \bdn_i(\prod_{j<d} O_i) = \bdn_i(\prod_{j<d}
U_i)\leq \bdn_i(U) \leq \bdn_i(\Aa^d).\]

Now, since \(\dim(X) = d\), there exists an \(\cL_i(M_i)\)-definable finite to one map \(f : X \to \Aa^d\) whose image is Zariski dense and hence has non-empty interior. So, we have \(\bdn_i(X) = \bdn_i(f(X)) = \bdn_i(\Aa^d)\).
\end{proof}

\begin{proposition}\label{lemsq}
Let $A =\acl(A)\subseteq K$, \(X\) be quantifier free \(\cL_I(A)\)-definable and
non-empty, $\phi(x,y)$ be an $\Ld$-formula and $(a_j)_{j\in \omega}\in K^y$ an
indiscernible sequence over $A$. Suppose that \(\phi(K,a_0)\) is $\tau$-dense in
$X$. Then \(\bigwedge_j \phi(x,a_j)\) is consistent.
\end{proposition}

\begin{proof}
By compactness, extend \((a_j)_{j\in\omega}\) to an \(A\)-indiscernible sequence
indexed by some sufficiently large cardinal \(\kappa\). For every finite
\(J\subseteq \kappa\), let \(a_J\) enumerate \(\acl(Aa_j \mid j\in J)\) --- in a
way compatible with inclusions: if \(f : J_1 \to J_2\) is an increasing
bijection, then for any \(J\subseteq J_1\), \(a_{J}\) and \(a_{f(J)}\) appear
in the same place in \(a_{J_1}\) and \(a_{J_2}\) respectively.

\begin{claim}\label{ind over gen}
There exists \(c\in X\) such that \(c\aindep_A \bigcup_{j<\omega} a_j\) and, for all finite \(J\subseteq \omega\), \(\tp(a_J/Ac)\) only depends on \(|J|\).
\end{claim}

\begin{proof}
Note that, by invariance, the Zariski closure of \(X\) is over \(A\) and hence we can find \(c\in X\) with \(c\aindep_A \bigcup_j a_j\). By iterated applications of Erdös-Rado and at the cost of reducing \(\kappa\) (exactly as when extracting an indiscernible sequence), we find \(J_0 \subseteq \kappa\), with order type \(\omega\), such that for all finite \(J \subseteq J_0\), \(\tp(a_{J}/Ac)\) only depends on \(|J|\). By homogeneity (and changing \(c\)), we may assume that \(J_0 = \omega\). 
\end{proof}

Note that if \(J_1,J_2 \subseteq \omega\) are disjoint of the same cardinality,
if \(e \in a_{J_1}\cap a_{J_2}\), then there is some \(n,m\) with \(a_{J_1,n} =
e = a_{J_2,m}\). Spreading \(J_1\) and \(J_2\) apart, we can find \(J_3\) such
that \(J_1\cup J_2\), \(J_1\cup J_3\) and \(J_2\cup J_3\) are ordered in the
same way. It follows that \(a_{J_1,n} = a_{J_3,m} = a_{J_2,n}\). So \(e\)
appears in the same place in all \(a_J\), with \(J\) of the given cardinality,
and \(\tp(a_J/Ace)\) only depends on \(|J|\). So we may assume that for disjoint
\(J_1,J_2 \subseteq \omega\), \(a_{J_1}\cap a_{J_2} = A\). 

\begin{claim}
There exists \(c \in \phi(K,a_0)\) such that \(c\aindep_A \bigcup_{j<\omega} a_j\) and, for all finite \(J\subseteq \omega\), \(\qftp(a_{J}/Ac)\) only depends on \(|J|\).
\end{claim}

\begin{proof}
Let \(\Sigma_i(x)\) be the (closure under consequence of the) partial type
expressing that \(\tp_i(a_{J}/Ax)\) only depends on \(|J|\). By compactness, it
suffices, given \(\psi_i(x) \in \Sigma_i(x)\) for every \(i\) and \(U \subseteq
V\) Zariski open over \(A(a_j)_{j<\omega}\), to find \(e \in  U\) realizing
\(\phi(x,a_0)\) and the \(\psi_i\). 

Let  \(c\) be such as in \cref{ind over gen} and let \(V\) be its algebraic locus over \(A\), which is also its algebraic locus over \(Aa_{j<\omega}\). Then \(c\) is in the \(\tau_i\)-interior of \(\psi_i(K)\) in \(V\) and, since \(c\in \psi_i(K)\cap X\), \(\psi_i(K)\) has non-empty \(\tau_i\)-interior in \(X\).  The existence of \(e\) as above now follows from the \(\tau\)-density of \(\phi(K,a_0)\) in \(X\) and the fact \(U\) is \(\tau\)-open (in \(V\) and hence in \(X\)). 
\end{proof}

Note that if \(E\) is an \(\cL_i(Ac)\)-definable finite equivalence relation, for some finite \(J_1,J_2 \subseteq \omega\) we have \(a_{J_1} E a_{J_2}\) and hence this holds for all \(J_1,J_2\). In other words, \(\tp_i(a_{J}/\acl_i(Ac))\) only depends on \(|J|\) and hence, since \(\acl(Ac) \subseteq \alg{A(c)}\subseteq\acl_i(Ac)\), for all \(i\), \(\qftp(a_{J}/\acl(Ac))\) only depends on \(|J|\). Let \(d\) enumerate \(\acl(Ac)\).

\begin{claim}
For all \(n>0\), there exists \(d_n \qfequiv_{a_{\{0,\ldots,2^n-1\}}} d\) such that \(K \models \bigwedge_{j<2^n}\phi(d_n,a_{j})\).
\end{claim}

\begin{proof}
We proceed by induction on \(n\) and let us assume we have found \(d_n\) ---
note that \(d_0 = d\) works. Let \(a^1 = a_{\{0,\ldots,2^n-1\}}\) and \(a^2 =
a_{\{2^n,\ldots,2^{n+1}-1\}}\), \(d^1 = d_n\) and \(d^2\) be such that \(d^2a^2
\equiv_A d^1a^1\), in particular, \(d^2\models \bigwedge_{2^n\leq
j<2^{n+1}}\phi(x,a_{j})\). Then \(a^1 \cap a^2 = A\), \(d \aindep_A a^1a^2\),
\(d^1 \equiv_A d^2\), \(d^1a^1 \qfequiv_A d a^1\) and \(d^2a^2 \equiv_A d^1a^1
\qfequiv_A d a^1 \qfequiv_A d a^2\). By \cref{Amalgamation}, we find \(d_{n+1}\)
such that \(d_{n+1} a^j \equiv_A d^j a^j\) --- and hence \(d_{n+1} \in
\bigcap_{j<2^{n+1}} \phi(K,a_j)\) ---  and \(d_{n+1}\qfequiv_ {a^1a^2} d\).
\end{proof}

So \(\bigwedge_j \phi(x,a_j)\) is indeed consistent.
\end{proof}

\begin{prop}\label{strong}
Let \(X\) be \(\Ld(K)\)-definable of dimension \(d\). If for every \(i\in I\),
\(\bdn_i(\Aa^d) \leq \kappa_i\) for some cardinal \(\kappa_i\), then
\[\bdn(X) \leq \sum_{i\in I} \kappa_i.\]
Moreover, if \(\bdn_i(\Aa^d) < \kappa\), where \(\kappa\) is an infinite cardinal, then
\[\bdn(X) < \sum_{i\in I} \kappa.\]
\end{prop}

\begin{proof} We proceed by induction on \(d\) and assume \(K\) is sufficiently saturated. Let \(\phi_l(x,a_{lj})_{l<\kappa,j<\omega}\) be an inp-pattern of depth \(\kappa\) in \(X\) --- where \(\kappa > \sum_i \bdn_i(\Aa^d)\) in the first case. Let \(A =\acl(A) \leq K\) be such that the \(\phi_l\) and \(X\) are all over \(A\). We may assume that the \((a_{lj})_{j<\omega}\) are mutually indiscernible over \(A\). Let \(V\) be the Zariski closure of \(X\). 

By \cite[Lemma\ 7.1]{Che-NTP2} (see \cref{burden}(2)), \cref{Density} and indiscernability, we find (uniformly) \(\cL_i(a_{lj})\)-definable \(Y_{lji}\subseteq V\) such that \(\phi_l(K,a_{lj})\) is \(\tau\)-dense in \(Y_{lj} := \bigcap_i Y_{lji}\). If any \(Y_{l_0j_0i_0}\) has dimension smaller than \(d\), then, by induction, \(\bdn(Y_{l_0j_0i_0}) \leq \sum_{i\in I} \lambda_i\), where \(\lambda_i + 1 = \kappa_i\). However, the \(\phi_l(x,a_{lj})_{l\neq l_0,j<\omega}\) is an inp-pattern of depth \(\lambda\) in \(Y_{l_0j_0i_0}\), where \(\lambda + 1 = \kappa\). So \(\kappa = \lambda +1 \leq \bdn(Y_{l_0j_0i_0}) + 1 \leq \sum_{i\in I} \lambda_i +1 \leq \sum_{i\in I} \kappa_i < \kappa\), a contradiction. In the second case, by induction, we would have \(\kappa \leq \bdn(Y_{l_0j_0i_0}) < \kappa\), which is also a contradiction.

So we may assume that all the \(Y_{lji}\) have dimension \(d\). In particular, they have non-empty \(\tau_i\)-interior in \(V\). Since the border of \(Y_{lji}\) in $V$ has lower dimension, by \cite[Lemma\ 7.1]{Che-NTP2} and induction, we may further assume that the \(Y_{lji}\) are \(\tau_i\)-open in \(V\).

Since \(\phi_l(x,a_{lj})\) is an inp-pattern, for all $f: \kappa \to \omega$ we have that \(\bigwedge_{l} Y_{l,f(l),i}\) is consistent. It follows that there are at most \(\kappa_i\) many \(l\) (respectively strictly less than \(\kappa\) in the second case) such that \(\bigwedge_l Y_{lji}\) is inconsistent. In both cases, it follows that we can find an \(l\) such that, for all \(i\), \(\bigwedge_j Y_{lji}\) is consistent.

Since the \(Y_{lji}\) are \(\tau_i\)-open (in \(V\)), by compactness,
\(\bigwedge_j Y_{lji}\) contains some \(\cL_i(K)\)-definable \(\tau_i\)-open
\(U_i \subseteq V\). Let \(U = \bigcap_i U_i\). Then, for every \(j\),
\(\phi(K,a_{lj})\) is \(\tau\)-dense in \(U\). Let \(B\supseteq A\) be such that
\(U\) is quantifier free \(\cL_I(B)\)-definable. We may assume that \(a_{lj}\)
is \(B\)-indiscernible. By \cref{lemsq}, \(\bigwedge_j \phi(K,a_{lj})\) is
consistent --- contradicting that \(\phi_l(x,a_{lj})_{l<\kappa,j<\omega}\) is an
inp-pattern.
\end{proof}

We also have the following qualitative corollary:

\begin{cor} \label{PTCNTP2}
Suppose that, for all $i \in I$, $T_i$ is $NTP_2$
(respectively strong), then any perfect $K\models\bSPTC$ also is.
\end{cor}

The bound in \cref{strong} is actually tight (although the argument is
surprisingly more involved than we expected at first).

\begin{lemma}\label{nice pattern}
If, for some \(i\in I\), \(\bdn_i(\Aa^d) \geq \kappa\) for some cardinal \(\kappa\), then there exists an inp-pattern \((\phi_l(x,a_{lj}))_{l<\kappa,j<\omega}\) of depth \(\kappa\) in any \(\tau_i\)-open \(\cL_i(K)\)-definable set \(X\subseteq \Aa^d\)
such that \(a_{lj}\in K\) and \(\phi_l(x,a_{lj})\) defines a \(\tau_i\)-open set.
\end{lemma}

\begin{proof}
Note first that, at the cost of reducing \(X\), we may assume that \(X = \prod_{k < d} X_k\) is a product of open (bounded) subsets of \(\Aa^1\). By \cref{bdni}, we find \(M_i\models T_{i,K}\) and \((\phi_l(x,a_{lj}))_{l<\kappa,j<\omega}\) an inp-pattern of depth \(\kappa\) in \(X(M_i)\). Writing \(\phi(M_i,a_{lj})\) as the union of its interior and a lower dimension set, by \cref{burden}.(2), we may assume that either \(\phi(M_i,a_{lj})\) is \(\tau_i\)-open or \(\dim(\phi(M_i,a_{lj})) < d\). In the latter case, it follows by \cref{burden}.(6), that \(\bdn(\Aa^{d-1})\geq \lambda\), where \(\lambda + 1 = \kappa\). By induction, we find an inp-pattern of depth \(\lambda\) in \(\prod_{k<d-1} X_k\) with \(\psi_l(M_i,c_{lj})\) consisting of \(\tau_i\)-open sets. Together with a uniformly definable family of disjoint \(\tau_i\)-open subsets of \(X_{d-1}\), they yield an inp-pattern of depth \(\kappa\) in \(X\).

So we may assume that the \(\phi_l(M_i,a_{lj})\) are \(\tau_i\)-open and there only remains to prove that we can choose the \(a_{lj}\) in \(K\).
Let \(c_{lj}\) be generic in some \(\Aa^d\) such that \(a_{lj} \in \acl(c_{lj})\). We may assume that \(a_{lj}\) enumerates \(\acl_i(c_{lj})\), that \((c_{lj}a_{lj})_{j}\) are mutually indiscernible sequences and that it is middle part of a sequence indexed by \(\omega + \omega + \omega^\star\). Let \(A_l\) consist of the union of \(K\), the \(\omega\) initial segment and the \(\omega^\star\) final segment of the sequence. For each j, let \(E_{lj}\) be the set of \(\cL_i(A_lc_{lj})\)-conjugates of \(a_{lj}\).

\begin{claim}\label{complete E}
For every \(j_0<\omega\), \(E_{lj_0}\)  is a complete type over \(A_l(c_{lj})_{j}(a_{lj})_{j\neq j_0}\).
\end{claim}

\begin{proof}
Assume there exists \(E \subset E_{lj}\) non-empty defined over
\(A_l(c_{lj})_{j}(a_{lj})_{j\neq j_0}\). Then, by indiscernability, moving the
parameters to the initial and final segment, such a set also exists with
parameters in A.
\end{proof}

Let \(X_{lj} = \bigcup_{a\in E_{lj}} \phi_l(x,a)\). If the \(X_{lj}\) are
consistent, then, for every \(k\) and \(j<k\) there exists \(e_{lj} \in E_{lj}\)
such that \(\bigwedge_{j<k} \phi_l(x,e_{lj})\) is consistent. But, by
\cref{complete E}, \(\prod_{j<k} E_{lj}\) is a complete type. This would imply
that, \(\bigwedge_{j<k} \phi_l(x,a_{lj})\) is consistent, contradicting that the
\(\phi_l(x,a_{lj})\) form an inp-pattern. It follows that the \(X_{lj}\) form an
inp-pattern (of depth \(\kappa\) in \(X\)).

So, at the cost of allowing \(\phi_l(x,y)\) to be an \(\cL_i(A)\)-formula for
some \(A\subseteq M_i\) containing \(K\), we may assume the \(a_{lj}\) to be
generic in some \(\Aa^m\) (over \(A\)). Let \(C\) be a transcendence basis of
\(A\) over \(K\) and \(b\in A\subseteq\acl_i(C)\) be such that \(\phi_l(x,y) =
\psi(x,y,b)\), where \(\psi\) is an \(\cL_i(K(C))\)-formula. As above, possibly
enlarging \(C\), we may assume that the set \(E\) of \(\cL_i(K(C))\)-conjugates
of \(b\) is a complete type over \(K(C,c_{lj})\). Then, for every \(e\in E\),
the \(\psi(x,a_{lj},e)\) are inconsistent. It follows that the
\(\theta(x,a_{lj}) = \bigvee_{e\in E} \psi(x,a_{lj},e)\) are inconsistent and
hence form an inp-pattern (of depth \(\kappa\) in \(X\)). So we may assume that
\(A = K(C)\) is generated over \(K\) by a (finite) set of algebraically
independent elements. Then \(C(a_{lj})_{l<\kappa,j<\omega}\) is a tuple of
algebraically independent elements over \(K\), in other words any finite
sub-tuple is \(K\)-generic in the affine space of the right dimension. It
follows, \emph{e.g.} by \cref{equiv SPTC}.(ii), that its \(\cL_i(K)\)-type can
be realized in \(K\).
\end{proof}

\begin{proposition}
Let \(X\) be \(\Ld(K)\)-definable of dimension \(d\).
If, for every \(i\in I\),  \(\bdn_i(\Aa^d) \geq \kappa_i\) for some cardinal
\(\kappa_i\), then
\[\bdn(X) \geq \sum_{i\in I} \kappa_i.\]
\end{proposition}

\begin{proof}
By \cref{burden}.(6), we may assume that \(X\subseteq \Aa^d\). By \cref{Density2}, we may further assume that \(X\) is \(\tau\)-dense in \(\bigcap_i X_i\) where \(X_i\) is \(\cL_i(K)\)-definable and \(\tau_i\)-open. By \cref{nice pattern}, assuming that \(K\) is sufficiently saturated, for every \(i\in I\), we find an inp-pattern \(\phi_{i,l}(x,a_{i,l,j})_{l<\kappa_i,j<\omega}\) of \(\tau_i\)-open sets in \(X_i\),with \(a_{i,l,j} \in K\).

Note that for any finite \(Y_i \subseteq \kappa_i\) and \(f_i : Y_i \to \omega\), \(\bigcap_{l\in Y_i} \phi_{i,l}(M_i,a_{i,l,f(i,l)})\) is \(\tau_i\)-open and hence, by \(\tau\) density, \(X \cap \bigcap_{i,l\in Y_i} \phi_{i,l}(M_i,a_{i,l,f(i,l)})\neq \emptyset\). It follows that together, the \(\phi_{i,l}(x,a_{i,l,j})\) form an inp-pattern of depth \(\sum_i \kappa_i\) in \(X\).
\end{proof}

We have thus proved the following equality:

\begin{theorem}\label{bdn PTC}
For every perfect \(K\models\bSPTC\) and
\(\Ld(K)\)-definable set \(X\) of dimension \(d\), we have
\[\bdn(X) = \sum_i \bdn_i(\Aa^d).\]
\end{theorem}

The equality holds in \(\mathrm{Card}^*\), \emph{cf.} \cite[Definition~1.29]{Tou-bdn}.

\begin{corollary}
Let \(\K\models \PAC\) be perfect and bounded and \(v_i\) be \(n\) independent
non-trivial valuations on \(K\). Then
\[\bdn(K, v_1,\ldots,v_n) = n.\]
\end{corollary}

A similar corollary holds for \(\PRC\) and \(\PpC\) fields, but we refer the
reader to the more general statement in \cref{burden PIC}.

\section{Forking and dividing}

In this section we will study the relationship between forking (and dividing) in
the theories $T_i$ and in the theory $\bSPTC$.

Let us recall our conventions and notation. We assume that \(I\) is finite and
that \(T = \bigvee_{i\in I} T_i\). We also assume that \textbf{(H$_i$)} holds
for every \(i\in I\) (see Hypothesis \ref{H2i}). We denote by \(\tau\) the topology generated by the \(\tau_i\), and we fix \(K\models\bSPTC\)
perfect sufficiently saturated and homogeneous. We fix some \(M_i\models
T_{i,K}\).

\begin{proposition}
\label{Fork=Div}
Let  $A = \acl(A) \leq K$ and $X$ be a $\Ld(K)$-definable subset of $K$ such
that $X \subseteq \bigcap_{i}X_{i}$ is \(\tau\)-dense, where $X_{i}$ is
$\cL_i(K)$-definable and $\tau_i$-open. Then $X$ does not divide over $A$ if and
only if for all $i$, $X_i$ does not divide over $A$ (in
\(M_i\)). 
\end{proposition}

\begin{proof}
$(\Leftarrow)$ We assume that, for every \(i\), \(X_i := \theta_i(M_i,b)\) does
not divide over \(A\) in \(M_i\), and we write $X$ as $\psi(K, b)$. Let $\kappa$
be a sufficiently large cardinal and $(b_j)_{j <\kappa}$ be an indiscernible
sequence over $A$ with $b_0 = b$. Let \(X_{ij} = \theta_i(M_i,b_j)\). By
indiscernability \(X_j = \psi(K,b_j) \subseteq \bigcap_{i}X_{ij}\) is
\(\tau\)-dense. Since \(X_i\) does not divide over \(A\), we have \(\bigcap_{j}
X_{ij} \neq \emptyset\) and, by compactness, it contains some non-empty
\(\cL_i(K)\)-definable \(\tau_i\)-open \(U_i\). Enlarging \(A\) and taking an
appropriate subsequence of the \(b_j\), we may assume that \(U_i\) is
\(\cL_i(A)\)-definable.

Note that since \(U := \bigcap_i U_i\) is \(\tau\)-open, \(X\) is \(\tau\)-dense
in \(U\). It follows, by \cref{lemsq}, that \(\bigcap_j (X_j \cap
U)\neq\emptyset\), and hence \(X\) does not divide over \(A\).

$(\Rightarrow)$ We now assume that, for some fixed \(i\), $X_i:= \theta_i(M_i,
b)$ divides over A and let $(b_j)_{j < \omega}$ be an $A$-indiscernible sequence
such that $b_0 = b$ and $\bigwedge_j\theta(x, b_j)$ is inconsistent. We may assume
that the \(b_j\) are algebraically independent over \(A\). For all \(j\),
\(A(B_j) \simeq A(b) \leq K\). Since $A\leq A(b_j)$ is regular, we have
$A(b_{j+1}) \ind^l_A A(b_0 \ldots b_j)$ and, by \cref{FA Ti:reg} and induction,
we see that $A(b_j: j < \omega) \models T_{i,\forall}$. By \cref{better
SPTC:reg}, \(K \leq K(b_j: j < \omega) \simeq K\otimes_A A(b_j: j < \omega)\) is
\(\cL_I\)-existentially closed, so we may assume that the \(b_j\) are in \(K\).
It follows that \(X_i\cap K\) divides over \(A\) in \(K\), and hence, so does
\(X \subseteq X_i\cap K\).
\end{proof}

\begin{corollary}
Let $a$ be a tuple of $K$ and $A \subseteq B \subseteq K$ such that $A= \acl(A)$
and $B= \acl(B)$.  Then $\tp(a/B)$ does not divide over $A$ if and only if, for
all $i \in \{1, \ldots, n\}$, $\tp_i(a/B)$ does not divide over $A$ (in $M_i$).  

\end{corollary}
\begin{proof}
This follows from \cref{Fork=Div} and \cref{Density2}.
\end{proof}
 
\begin{thm} \label{Extbase}
Let $A = \acl(A) \leq K$. If $A$ is an extension base of $T_i$, for all $i \in I$, then $A$ is an extension base of $K$.
\end{thm}

\begin{proof}
By transitivity of forking, it suffices to show that no \(\Ld(A)\)-definable set
\(X\subseteq K\) forks over \(A\). So, if \(X \subseteq \bigcup_{j<m} Y_j\), where the
\(Y_j\) are \(\Ld(K)\)-definable, we have to show that one of the \(Y_j\) does
not divide over \(A\). To do so, we may assume:
\begin{enumerate}[(i)]
\item \(X \subseteq \bigcap_i X_i\) is \(\tau\)-dense, where the \(X_i\) are
\(\cL_i(A)\)-definable.
\item \(X \setminus \bigcup_j Y_j\) is finite.
\item \(Y_j \subseteq \bigcap_i Y_{ji}\) is \(\tau\)-dense, where
\(Y_{ji}\subseteq X_i\) is \(\tau_i\)-open and \(\cL_i(K)\)-definable.
\item For every \(i\) and \(j_1, j_2\), \(Y_{j_1 i} = Y_{j_2 i}\) or
\(Y_{j_1i}\cap Y_{j_2i} = \emptyset\).
\item \(X_i \setminus \bigcup_{j} Y_{ji}\) is finite and does not contain any
element of \(A\).
\end{enumerate}
Indeed, by \cref{Density2}, \(X\) is \(\tau\)-dense in \(\bigcup_l\bigcap_i
X_{il} \cup Z\) where the \(X_{il}\) are \(\cL_i(A)\)-definable and \(Z\) is
finite and \(\Ld(A)\)-definable. If \(Z\) is non-empty, it does not fork over
\(A\), and we are done. Otherwise, it suffices to show that any of the \(X \cap
\bigcap_{i} X_{il}\) does not fork over \(A\). Similarly, by \cref{Density2},
\(Y_j\) is \(\tau\)-dense in \(\bigcup_l\bigcap_i Y_{jil} \cup Z_j\) where the
\(Y_{jil}\) are \(\cL_i(K)\)-definable \(\tau_i\)-open and \(Z_j\) is finite.
Replacing the \(Y_{j}\) by the \(Y_{j}\cap \bigcap_i (Y_{jil}\cap X_i)\) and
setting the finite sets aside, we can ensure (ii) and (iii).

Replacing the \(Y_{ji}\) by the \(\tau_i\)-interior of the atoms of Boolean
algebra that they generate (for fixed \(i\)), we can further assume (iii). Now,
for some finite set \(Z\), we have \(X\subseteq \bigcup_j Y_{ji} \cup Z
\subseteq X_i\) and since \(X\) is \(\tau_i\)-dense in \(X_i\), \(Z_i :=
X_i\setminus(\bigcup_j Y_{ji} \cup Z)\) has empty \(\tau_i\)-interior and is
therefore finite. Now, if \((Z_i\cup Z) \cap A \neq \emptyset\), we can remove
those points from \(X\) and the \(X_i\) without losing any of the other
properties.

There remains to show that one of the \(Y_j\) does not divide over \(A\). Note
first that, for every \(f : I \to m\), by \(\tau\)-density, there exists an
\(a\in X\cap \bigcap_i Y_{f(i) i}\). Let \(j\) be such that \(a\in Y_j \subseteq
\bigcap_i Y_{ji}\). Then, by (iv), for all \(i\) we have \(Y_{f(i) i} =
Y_{ji}\).

Now, by hypothesis, \(X_i\) does not fork over \(A\) and since any finite set
that does not divide over \(A = \acl(A)\) must contain a point from \(A\), there
is an \(f(i) < m\) such that \(Y_{f(i) i}\) does not divide over \(A\). As noted
above, we may assume that \(f(i) = j\) is constant. It now follows from
\cref{Fork=Div}, that \(Y_{j}\) does not divide over \(A\).
\end{proof}

\begin{definition}
Let $a$ be a tuple of $K$ and  $A , B\subseteq K$. We write $a \ind_A B$ if
$\tp(a/AB)$ does not fork over $A$ (in $K$) and $a \ind^i_A B$ if
$\tp_{i}(a/AB)$ does not fork over $A$ (in $M_i$).
\end{definition}

\begin{corollary}\label{char fork}
Suppose that $T_i$ is NTP$_2$ for all $i\in I$ and that forking equals dividing
in $T_i$. Then forking equals dividing in $K$ and if $a, A, B$ are in $K$, then
$a \ind^i_A B$ for all $i\in I$ if and only if $a \ind_A B$. 
\end{corollary}
\begin{proof}
Since $K$ is NTP$_2$ (\cref{PTCNTP2}) to show that forking equals dividing in $K$ it's enough to show that any set in $K$  is an extension base, but this is trivial by \cref{Extbase}.

The last equivalence is clear from \cref{Fork=Div} and the fact that forking equals dividing in $K$.
\end{proof}

\appendix

\section{Dependent topologies}

Following the approach of \cite{Joh-MulACVF} we can also understand dependent
valuations and orders by reduction to the case of independent topologies.

We now fix a finite tree \(I\) --- that is a filtered partial order such that
for all \(i\in I\), the set \(\{j\in I\mid i \leq j\}\) is totally ordered. Let
\(r\) be its maximal element --- its root. To every leaf \(i\) --- that is every
minimal element --- we associate:
\begin{itemize}
\item An integer \(p_i\) which is either prime or zero.
\item A theory complete \(T_i\) of fields of characteristic \(p_i\)  which
eliminates quantifiers in a relational expansion of the ring language such that
\textbf{(H$_{i}$)} holds (see hypothesis \ref{H2i}).
\end{itemize}

To every non-leaf \(i\in I\), we associate:
\begin{itemize}
\item An integer \(p_i\) which is either prime or zero. We assume that for all
daughters \(j\) of \(i\), the \(p_j\) are equal. Moreover, they are equal to
\(p_i\) if it is non-zero. We also assume that \(p_r\) is equal to its daughter \(p_j\).
\item A theory complete \(T_{i,\Gamma}\) of (enriched) ordered Abelian group. We
assume that it is the theory of the trivial group if and only if \(i = r\). If
\(p_i\) is non-zero, we assume that models of \(T_{i,\Gamma}\) are
\(p_i\)-divisible. If \(p_i\) is zero and daughter \(p_j\) are non-zero, we
require \(T_{i,\Gamma}\) to come with a constant \(c_i\) for a positive element.
We then assume that\([-c_i,c_i]\) is either \(p_j\)-divisible or finite --- in
the latter case, we say that \(i\) is finitely ramified.
\end{itemize}

For every non-leaf \(i\in I\), we define the theories \(T_{i,\k}\)
and \(T_i\), by induction:
\begin{itemize}
\item The theory \(T_{i,\k}\) (the residual theory of \(i\)) is the theory of
pseudo \(I_i\)-closed fields without finite extensions of degree divisible by
\(p_i\) (if it is non-zero), where \(I_i = \{T_j : j\) daughter of \(i\}\). 

\item The theory \(T_i\) is a Morleyization of the theory of characteristic
\(p_i\) algebraically maximal complete fields with residue field a model of
\(T_{i,\k}\) and value group a model of \(T_{i,\Gamma}\). If \(p_i\) is zero and
the daughter \(p_j\) are non-zero, then we also require that \(c_i = \val(p_j)\).
\end{itemize}

Note that, by \cite[Corollary A.3]{HalHas-QE}, if \(i\) is not finitely ramified,
the reduct of \(T_i\) to the valued field language eliminates quantifiers
relative to \(\RV_1\); and it eliminates quantifiers relative to \(\bigcup_n
\RV_n\) if \(i\) is finitely ramified, by \cite[Theorem~B]{Bas-QE}. Moreover,
each of completion admits maximally complete models. In particular,
\textbf{(H$_{i}$)} holds in every \(T_i\).

\begin{definition}
We say that a field is pseudo \(I\)-closed if it is pseudo \(I_r\)-closed.
\end{definition}

This is coherent with our earlier notation, provided we identify a set of
independent valuations to the tree of height one with trivial valuation at the
root, and the independent valuations as leaves.

We also define:
\[\kappa_i = \left\{\begin{array}{ll}
\bdn(T_i) & \text{if \(i\) is a leaf,}\\
\max(\aleph_0 \cdot \sum_{j\in I_j} \kappa_j,\bdn(T_{i,\Gamma})) & \text{if \(i\) is finitely ramified,}\\
\max(\sum_{j\in I_j} \kappa_j,\bdn(T_{i,\Gamma}))& \text{otherwise.}
\end{array}\right.\]

Once again, for a finer estimate, \(\kappa_i\) can be computed in \(\mathrm{Card}^\star\), \emph{cf.} \cite[Definition~1.29]{Tou-bdn}.

\begin{theorem}
Let \(K\) be a bounded perfect pseudo \(I\)-closed field. Then
\[\bdn(K) = \kappa_r.\]
\end{theorem}

\begin{proof}
We proceed by induction on \(i\). By \cref{bdn PTC}, \(\bdn(T_{i,\k}) =
\sum_{j\in I_i} \kappa_i\) (by \cref{bdn PTC}) and, by \cref{bdn VF},
\(\bdn(T_i) = \kappa_i\).
\end{proof}

\begin{corollary}\label{burden PIC}
Let \(K\) be a bounded perfect \(\PAC\), \(m\)-\(\PRC\) or \(m\)-\(\PpC\) field
(setting \(m=0\) if \(K\) is \(\PAC\)) and \(v_1,\ldots,v_l\) be valuations on
\(K\). Let \(n\) be the maximal number of non-dependent valuations among the
\(v_i\), that are also not dependent from the \(m\) orders (respectively
\(p\)-adic valuations) if \(K\) is \(\PRC\) (respectively \(\PpC\)). Then
\[\bdn(K,v_1,\ldots,v_l) = m + n.\]
\end{corollary}

\begin{proof} 
We assume \(K\) is \(\PAC\). Let \(I\) be the tree of valuations generated by
the \(v_i\), in other words the nodes of \(I\) are the valuation rings
\(\prod_{i\in J} \Val_i\), where \(J\subseteq I\) and \(\Val_i\) is the
valuation ring of \(v_i\) (and the trivial valuation at the root). For any node
\(i\in I\) we also write \(v_i\) for the associated valuation. Since \(K\) is
\(\PAC\), the residue field of any of the \(v_i\) is algebraically closed and
its value group is divisible --- this follows, for example, from \cref{PTC
dense}. We annotate \(I\) with a copy of \(\ACVF\) on each leaf and with the
theory of divisible ordered Abelian groups on each non-leave node (except for
the root), as well as the relevant characteristics and constants.

It follows from \cref{free ec} that \(K\models\PIC\) and this structure is
bi-interpretable with \((K,v_1,\ldots,v_n)\). By induction on \(i\) we see that
\(\kappa_i\) is the number of leaves below the node \(i\). The corollary
follows.

The \(\PRC\) and \(\PpC\) cases follow in a similar manner, noting that
orders (respectively \(p\)-adic valuations) can only be leaves in the tree.
\end{proof}

\sloppy
\printbibliography

\end{document}